\documentclass[a4paper, 11pt]{amsart}
\usepackage{amsmath,amssymb}
\usepackage[T1]{fontenc}
\usepackage{mathtools}

\DeclareMathOperator{\supp}{{\rm supp}}

\newtheorem{theorem}{Theorem}
\newtheorem{lemma}{Lemma}[section]
\newtheorem{conjecture}{Conjecture}
\newtheorem{proposition}{Proposition}

\newtheorem{definition}{Definition}
\newtheorem{claim}{Claim}[section]
\numberwithin{equation}{section}
\newcommand{\dt}{\partial_t}
\newcommand{\ds}{\partial_s}
\newcommand{\di}{\partial_i}
\renewcommand{\dj}{\ensuremath{\partial_j}} 
\newcommand{\dn}{\partial_n}
\newcommand{\dij}{\partial_{ij}}

\newcommand{\dyn}{\partial_{y_n}}
\newcommand{\dx}{\partial_x}

\newcommand{\grad}{\nabla}
\newcommand{\gradx}{\grad_{\negthickspace x}}
\newcommand{\grady}{\grad_{\negthickspace y}}
\newcommand{\gradsy}{\grad_{\negthickspace s,y}}
\newcommand{\lap}{\Delta}
\newcommand{\lapx}{\lap_x}
\newcommand{\lapy}{\lap_y}

\newcommand{\yzeron}{y_{0,n}}

\newcommand{\dist}{d}

\newcommand{\Rn}{{\mathbb{R}^ n}}
\newcommand{\R}{\mathbb{R}}
\newcommand{\RR}{\mathbb{R}}
\newcommand{\ren}{\mathbb{R}^n}
\newcommand{\ve}{\varepsilon}
\newcommand{\N}{\mathbb{N}}
\newcommand{\Pos}{\mathcal{P}}
\newcommand{\ldef}{\vcentcolon=}

\title[Flatness implies smoothness]{Flatness implies smoothness for solutions \\of the Porous Medium Equation}
\author{Clemens Kienzler}
\address{(CK) Mathematisches Institut der Universit\"at Bonn, Endenicher Allee 60, 53115 Bonn, Germany}
\email{kienzler@math.uni-bonn.de}
\author{Herbert Koch}
\address{(HK) Mathematisches Institut der Universit\"at Bonn, Endenicher Allee 60, 53115 Bonn, Germany}
\email{koch@math.uni-bonn.de}
\author{Juan Luis V\'azquez.}
\address{(JLV) Dpto. de Matem\'aticas, Univ. Aut\'onoma de Madrid, 28049 Madrid, Spain}
\email{juanluis.vazquez@uam.es}

\begin{document}

\begin{abstract}
  One of the major problems in the theory of the porous medium
  equation $\dt \rho=\lapx \rho^m, $ $ m > 1$, is the regularity of
  the solutions $\rho(t,x)\ge 0$ and the free boundaries
  $\Gamma=\partial\{(t,x): \rho>0\}$.  Here we assume flatness of the
  solution and derive $C^\infty$ regularity of the interface after a
  small time, as well as $C^\infty$ regularity of the solution in the
  positivity set and up to the free boundary for some time interval.

  We use these facts to prove the following eventual regularity
  result: solutions with compactly supported initial data are smooth
  after a finite time $T_r$ that depends on $\rho_0$. More precisely,
  $\rho^{m-1}$ is $C^\infty$ in the positivity set and up to the free
  boundary, which is a $C^\infty$ hypersurface for $t \ge
  T_r$. Moreover, $T_r$ can be estimated in terms of only the initial
  mass and the initial support radius. This result eliminates the
  condition of non-degeneracy on the initial data that has been
  carried on for decades in the literature. Let us recall that
  regularization for small times is false, and that as $t\to\infty$
  the solution increasingly resembles a Barenblatt function and the
  support looks like a ball.

\end{abstract}

\maketitle

\section{Introduction}\label{intro}

We consider the porous medium equation (PME), that we will write as
\begin{equation}\label{PME}
\dt \rho=k\lapx \rho^m, \qquad  m > 1\,.
\end{equation}
We assume that the solution is defined in a time-space cylinder $ Q
\ldef I \times \Omega $ with an open time interval $ I \ldef (t_1,t_2)
$ and an open set $ \Omega \subset \Rn $. This is convenient for the
local regularity theory. In the last part we consider global solutions
defined in \ $ Q \ldef (0,\infty) \times \RR^{n} $.

We introduce a constant $k>0$ for convenience in the calculations
though it is irrelevant in the results since it can be absorbed for
instance into the $\rho$ variable, without changing time or space. To
be precise, the change $\tilde \rho = k^{1/(m-1)}\rho $ \ allows to
pass from $ k > 0$ to the value $\tilde k=1 $ of the usual PME.  We
will fix the value \ $ k = \frac{m-1}{m}$ \ throughout the paper,
since in this way the explicit formulas that enter our computations in
the local regularity theory will be easier to read. In particular, the
equation can then be written as
\begin{equation}
\dt \rho=\nabla\cdot(\rho\nabla \rho^{m-1})\,,
\end{equation}
and we use the formula $ v=\rho^{m-1} $ to define the `pressure' and
then $v(t,x)$ satisfies
\begin{equation}
\dt v = (m-1) v \lapx v + |\gradx v|^2.
\end{equation}
It is well-known that the equation can be solved in a unique way in
the sense for instance of continuous weak solutions, after giving
Dirichlet data in the parabolic boundary of $Q$, \cite{Aro86,
  vazquezPME}, or just giving initial data if the space domain is
$\R^n$. It is also well-known that initial data that vanish say in the
closure of an nontrivial open set at the initial time $t_1$ will
vanish for a certain time interval $J= (t_1, t_0)$, $ t_0 \le t_2 $,
in a possibly smaller set. This property is called finite speed of
propagation. If we denote the positivity set of the solution by
$\mathcal P(\rho)=\{(t,x): \rho(t,x)>0\}$ and write the section at
time $t$ as
\begin{equation}\label{positivityset}
{\mathcal P}(t)=\{ x\in \Omega \mid \rho(t,x)>0 \}\,,
\end{equation}
then finite propagation means that ${\mathcal P}(t)$ is smaller than $
\Omega $ for all $t\in J$.  The consequence of finite propagation is
the existence of a {\sl free boundary} or {\sl interface}
$\Gamma=\partial {\mathcal P}\cap Q$, where the transition from
the `gas region' $\mathcal{P}$ to the `empty region' $\{\rho=0\}$ takes
place.

The study of the free boundary $\Gamma$ plays a key
role in the regularity theory for the PME \eqref{PME}, since the
equation is not parabolic  at those points.  The positivity set
$\mathcal{ P}(t)$ is open and increasing with respect to
$t$ for $ t \ge t_1$, and strictly increasing outside the original support.
Thus, the free boundary is given as a graph of
the form $t= h(x)$, where $h$ is a function defined in the closure of
the complement of the support of $\rho(t_1,\cdot)$. Let us consider
global solutions for simplicity.  It was shown in \cite{CaffFr} that
the function $ h $ is at least H\"older continuous.  As a consequence
of \cite{CVW87}, $h$ is Lipschitz continuous for $t$ sufficiently
large when the initial datum is a compactly supported and
non-degenerate function (roughly that the velocity $|\nabla
v|_{\partial P(t)}$ at the free boundary $\partial P(t)$ is bounded
  from below, see \cite{CVW87} for a nondegeneracy condition on the
  initial data ensuring non-degeneracy of the solution for large time,
  which was relaxed in \cite{Koc99}), and hence the solution is
  monotone in a cone of directions including the time direction
  locally near the free boundary.  A decade later, the second author
  \cite{Koc99} proved  that solutions satisfying this monotonicity and
  non-degeneracy conditions are actually smooth up to the free
  boundary.

  We recall that regularity has an important consequence
  for the dynamics of the free boundary $\Gamma$. Indeed, once you
  prove that $\Gamma$ is (at least locally) a smooth hypersurface in
  $\RR^{n+1}$ and the pressure is also smooth and non-degenerate on
  the occupied side $v >0$ and up to the free boundary, then standard
  theory shows that Darcy's law holds on that part of $\Gamma$, in
  both forms: $v_t=|\nabla v|^2$ or $V_n=|\nabla v|$, cf. \cite{vazquezPME} (where $V_n$
  denotes the normal advancing speed of the free  boundary,
  and both $v_t$ and $\nabla v$   are calculated as lateral limits from the region $\{v>0\}$).

  This paper was motivated by the wish to obtain a higher regularity
  theory for local solutions, and to eliminate the extra conditions of
  monotonicity and non-degeneracy on the initial data that were needed
  to establish the higher regularity of solutions and free boundaries
  in dimensions $n\ge 2$. Such result was known for $n=1$,
  cf. \cite{ArVaz87}. For years the efforts have been unsuccessful in
  several dimensions. In this paper we contribute a key step by
  considering {\sl locally flat solutions}, in the sense that they are
  sandwiched on a cylinder between two traveling wave solutions lying
  very close to each other, and then we show that they are smooth up
  to the free boundary in half of the cylinder (thus, a bit later in
  time). See Definition \ref{def.flat} below for the concept of
  $\delta$-flatness, and Theorem \ref{th-main} for the precise
  formulation of our main result.

Important steps in the direction of this result were taken in the
Dissertation of the first author, \cite{Kienz13}, which deals with the
special case of flat solutions on a global scale, and such results
will be used in the course of our project.  On the other hand, our
main result has as a consequence the result we were originally looking
for, concerning the large time behavior of global solutions (i.\,e.,
defined on the whole $\R^n$) with compactly supported initial data. As
a consequence of the well-known asymptotic convergence towards a
Barenblatt profile \cite{vazquez1}, such solutions satisfy our flatness
condition for large enough times. This allows to dispense completely
with the regularity and non-degeneracy assumptions used in
\cite{CVW87} and \cite{Koc99} on the initial data.

In particular, we will prove that the free boundary function $h$ above
is $C^\infty$ smooth outside a large ball, and it converges to the
free boundary of the Barenblatt solution with same mass, see Theorem
\ref{thm2}. Moreover, the pressure of the solution is $C^\infty$ in
the positivity set and up to the free boundary for times \ $t\ge
T(u_0)$, a lower bound that can be estimated in terms of the initial
mass and the initial radius. Much is known about the precise behavior
of solutions when $t\to\infty$.  Recently, \cite{Seis1, Seis2} has
obtained very precise asymptotics as $t\to \infty$ for solutions with
compact support.

\subsection{Preliminaries, definitions and main result.} Our work uses
many different tools of the PME theory, most of them can be found in
\cite{vazquezPME}. Two basic facts will have a special relevance, the use
of the scaling group and the existence of a family of {\sl traveling
  wave solutions}.

\noindent $\bullet$ The first one consists of the observation that,
due to the symmetries of the PME, whenever $\rho(t,x)$ is a solution
in a given cylinder, so is the expression
\begin{equation}\label{scaling.form}
\widetilde{\rho}(t,x)=A\,\rho(C\,(t-t_0), L\,(x-x_0))
\end{equation}
for any constants $A,C,L>0$ such that $A^{m-1}=CL^{-2}$, and all
displacement constants $x_0\in \ren$ and $t_0\in \RR$. Of course, the
domain of definition varies accordingly. We may add also a rotation or
symmetry of the space variables without affecting the validity of the
solution.

\noindent $\bullet$ The {\sl traveling waves } are a family of
solutions that are perfectly flat in terms of the pressure variable
$v=\rho^{m-1}$. They are given by the formula
\begin{equation}
  \rho_{tw}^{m-1}(t,x; {\bf a},c,d) \ldef c \, (c \, t  + \langle {\bf a}, x \rangle + d)_+,
\end{equation}
defined for $(t,x) \in \R \times \Rn $. There are a number of
parameters that can be fixed at will: the wave speed $ c > 0 $; $ {\bf
  a} \in \RR^{n} $ is a unit vector indicating the direction ($-{\bf
  a}$ would be more appropriate) in which the wave front propagates;
the displacement $ d \in \RR^{n} $ is also arbitrary and incorporates
possible displacements of the $x$ or $t$ axis or both. If $ {\bf a} $
points into the $ n $-th coordinate direction and the wave travels at
unit speed $ c = 1 $, the formula for the traveling wave at $ (t_0,
x_0) = (0,0) $ simplifies to
\begin{equation}\label{form-twc}
	\rho_{tw,d}(t,x) \ldef  (t  + x_n + d)_+^\frac{1}{m-1},
\end{equation}
We will drop the subscript when the displacement $ d = 0 $ .

\noindent $\bullet$ Next, we need a rough concept of flatness of a solution. For
  easier reference in more general applications we state it for any
  $a$, $c$, and some $d>0$.

\begin{definition}\label{def.flat}
  Consider $ (t_0,x_0) \in \R \times \Rn $ and $ \delta $, $ R > 0
  $. We say that $ \rho $ is a $ \delta $-flat solution of the PME at
  $ (t_0,x_0) $ on scale $ R $ if there exist $ V > 0 $ and a unit
  vector $ {\bf a} \in \Rn $ such that

\begin{enumerate}
\item $ \rho $ is a nonnegative and continuous weak solution of the $
  PME $ on the cylinder
	\[
		 Q_{R,V}(t_0,x_0) \ldef (t_0 - \frac{R}{V},t_0] \times B_R(x_0)
	\]
\item and
\begin{equation}
  \rho_{tw}(t - t_0,x - x_0; {\mathbf a},V,-\delta R) \leq \rho(t,x) \leq \rho_{tw}(t - t_0,x - x_0; {\mathbf a},V,\delta R)
\end{equation}
on $ Q_{R,V}(t_0,x_0) $.
\end{enumerate}
We call $ V $ the speed and ${\mathbf a} $ the direction of $ \rho $.
\end{definition}

In what follows we use the normalized values  $V=1$  and ${\bf a}=e_n$
unless mention to the contrary. Moreover, we should say a $ \delta
$-approximate speed and a $ \delta $-approximate direction to be
precise, since they are not unique for any given solution, they may
admit small variations.

Let us perform the scaling reduction.  Given a $ \delta $-flat
solution $ \rho $ at $ (t_0,x_0) $ on scale $ R $ with $ \delta
$-approximate speed $ V $ and $ \delta $-approximate direction $ a $,
we may use the two-parameter scaling group with parameters $R,V>0$ and
introduce any orthogonal matrix $ O $ with $ O e_n = {\bf a} $ to
obtain a function
\begin{equation} \tilde \rho(t,x) = (V R)^{-\frac1{m-1}} \rho\left(
    V^{-1} \, R \, t + t_0, R \, O x + x_0 \right)
\end{equation}
which is another solution of the PME (with the same $k$) and this
function is $ \delta $-flat at $ (0,0) $ on scale $ 1 $ with $ \delta
$-approximate unit speed and $ \delta $-approximate direction $ e_n $;
in other words, we have
\begin{equation}
  \rho_{tw,-\delta} \leq \tilde \rho \leq \rho_{tw,\delta} \quad \text{ in } Q_{1,1}(0,0)\,,
\end{equation}
meaning that $ \tilde \rho $ is trapped between two traveling wave
solutions with velocity $1$ that lie at a distance $ 2 \delta $ in the
unit cylinder centered at zero. Note that exactly one of the traveling
wave solutions is positive in $(0,0)$. We are now ready to state our
main contribution.

\begin{theorem}\label{th-main} There exists $ \delta_0>0$ such that the following holds: \\
  If $ \rho $ is a nonnegative $ \delta $-flat solution of the PME at
  $ (0,0) $ on scale $ 1 $ with $ \delta $-approximate direction $ e_n
  $ and $ \delta $-approximate speed $ 1 $, and $ \delta \leq
  \delta_0$, then for all derivatives we have uniform estimates
\begin{equation}
 | \dt^k \dx^\alpha \gradx (\rho^{m-1} -(x_n+ t)) | \le C \delta
\end{equation}
at all points $(t,x)\in ([-1/2,0] \times \overline{B(0,\frac12)})\cap
\mathcal P(\rho) $ with $C=C(n,m,k,\alpha)>0$. In particular,
$\rho^{m-1}$ is smooth up to the boundary of the support in
$(-\frac1{2},0]\times B_{1/2}$, and
\begin{equation}
|\gradx \rho^{m-1}-e_n |, \quad |\partial_t \rho^{m-1} - 1|  \le C \delta\,.
\end{equation}
Moreover, the level sets for positive values of $\rho$ and the free boundary are uniformly smooth hypersurfaces inside $(-\frac1{2} ,0]\times B_{\frac12}(0)$.
\end{theorem}

Theorem \ref{th-main} is formulated in a normalized setting. It can be
combined with the symmetries of the porous medium equation in the
obvious fashion to get a result in a not normalized setting.  We point
out that though our strategy follows the general ideas of the proofs
developed by Caffarelli for similar free boundary problems, we
considerably deviate from his arguments in the proof of Proposition
\ref{improves} below, mainly because the proper linearization of
geometry and density at the free boundary is a nonstandard degenerate
equation, the solutions of which must replace the use of harmonic and
caloric functions used in the standard theory of free boundary
problems.

Theorem \ref{th-main} implies the eventual $C^\infty$-regularity
result for global solutions that we have already mentioned as our
second contribution.  We use the notation
$R_B(t)=c_1(n,m)\,M^{(m-1)\lambda}t^{\lambda} $ with
\begin{equation}\label{lambda}
\lambda=1/(n(m-1)+2)
\end{equation}
for the Barenblatt radius for the solution with mass $M$ located at the origin.

\begin{theorem}\label{thm2} Let $\rho\ge 0$ be a solution of the PME
  posed for all $x\in \Rn$, $n\ge 1$, and $t>0$, and let the initial
  data $\rho_0$ be nonnegative, bounded and compactly supported with
  mass $M= \int \rho_0 dx>0$.  Then, there exists a time $T_r$
  depending on $\rho_0$ such that for all $t>T_r$ we have:

\begin{itemize}
\item[(i)] The pressure of the solution $\rho^{m-1}$ is a $C^\infty$
  function inside the support and is also smooth up to the free
  boundary, with $\nabla \rho^{m-1} \ne 0 $ at the free
  boundary. Moreover, the free boundary function $t=h(x)$ is
  $C^\infty$ in the complement of the ball of radius $R(T_r)$ and,
  there exists $c>0$ such that
\begin{equation} \label{precise}
  \begin{split}
    t^{-n\lambda} \Big(a^2 M^{2(m-1)\lambda}- c t^{-2\lambda} -
    \frac{\lambda |x-x_0|^2}{2 t^{2\lambda}} \Big)^{\frac{1}{m-1}}_+
    \le & \rho(t,x)\\ & \hspace{-4cm} \le t^{-n\lambda} \Big(a^2
    M^{2(m-1)\lambda}- c t^{-2\lambda} - \frac{\lambda |x-x_0|^2}{2
      t^{2\lambda}} \Big)^{\frac{1}{m-1}}_+
\end{split} \end{equation}
where   $x_0=M^{-1}\int x\rho(x)dx$  is the conserved center of mass, and $a$ is the constant defined in \eqref{defa}. Moreover,
\begin{equation} \label{precisesupport}
B_{R_B(t)-c t^{-\lambda}}(x_0) \subset \supp \rho \subset B_{R_B(t)+ct^{-\lambda}}(x_0)
\end{equation}
\item[(ii)] Moreover, if the initial function is supported in the ball
  $B_R(0)$, then we can write the upper estimate of the regularization
  time as \
\begin{equation}\label{lowerboundontime}
    T_r=T(n,m) M^{1-m}  R^{\frac1\lambda}.
\end{equation}
\end{itemize}
\end{theorem}

By scaling and space displacement we can reduce the proof to the case
$M=1$ and $x_0=0$.  We remind the reader that a minimum delay $T_r$ is
needed for general initial data, even under the assumptions of compact
support and smoothness, for the regularity result to hold. Indeed, it
is known that the initial regularity of a solution can be lost for
some intermediate times because of the phenomenon called focusing,
whereby typically a hole in the support of the initial data gets
filled in finite time by the evolving solution, and then $\rho^{m-1}$
is not Lipschitz continuous near the focusing point at the focusing
moment, cf. \cite{AA95, AG93}. But this phenomenon disappears in
finite time for compactly supported solutions as shown in
\cite{CVW87}, and eventually solutions and interfaces are $C^\infty$
smooth as the theorem claims. The asymptotics in part 1 are sharp, see
Seis \cite{Seis2}, who also discusses finer asymptotics.

We conclude this introduction with a conjecture.
\begin{conjecture} Suppose that $\rho$ is a solution to the porous
  medium equation on $[-1,0] \times B_2(0)$. We assume that there is a
  nonempty open cone $C$ so that $\rho(t,x) \le \rho(t,y)$ for $y-x\in
  C$. Then $\rho^{m-1}$ is smooth up to the boundary of the support
  outside the initial support and $\nabla \rho^{m-1} \ne 0$ at the
  free boundary.
\end{conjecture}

\subsection{Outline}
We gather in Section \ref{basicprops} the precise statements of the
main propositions that form the basis of the proof of Theorem
\ref{th-main}, we state and prove the main body of the proof of the
theorem in Sections \ref{sec.prop1} and \ref{proofthm}, and then we
prove the list of auxiliary propositions in the later
sections.  We derive Theorem \ref{thm2} in Section
\ref{sec.thm2} after the detailed quantitative analysis of the
evolution of global solutions and their interfaces  is  done.

We will use the following rather standard notations for the space-time
cylinders that appear, $Q_r(t_0,x_0)=(t_0-r, t_0)\times B_r(x_0)$ and
$Q_r=Q_r(0,0)$. The open space ball of center $x_0\in \Rn$ and radius
$r>0$ will be denoted by $B_r(x_0)$.


\section{Idea of the proof. Basic Propositions}
\label{basicprops}

The key step in the proof of Theorem \ref{th-main} is the following
self-improvement result.

\begin{proposition}[Improvement of flatness]\label{improves} There
  exist
  $\delta_0$ and $r$,
  \[ 0 < 2\delta_0 < r < \frac12
  \]
such that if $\rho $ is a solution  of the PME on $Q_1$ with $\rho(0,0)=0$,
 $(0,0)$ lies in the boundary  of the support, and $\rho $ is  $ \delta $-flat
with velocity $1$ and scale $1$ at $(0,0)$
 with a $\delta < \delta_0 $, then there exist $\Lambda>0$ and a unit vector
  $\mathbf{a}$ which satisfy
\[
|1-\Lambda | \le c\frac{\delta}{r}, \qquad |\mathbf{a} -e_n| \le c \sqrt{\delta}{r}\,,
\]
such that   the rescaled solution

\begin{equation}
\tilde \rho  (t,x) =  r^{-\frac1{m-1}}  \rho\Big(\Lambda^2 r t ,  \Lambda  rt  x \Big)
\end{equation}
 is $\delta/2$-flat with velocity $1$ and scale $1$ at $(0,0)$,  in direction $\mathbf{a}$.
\end{proposition}
To establish  these facts we prove and then use  a number of regularity results:
\begin{itemize}
\item  Proposition \ref{prop_global} for global solutions to the PME with initial  pressure which is Lipschitz  close to a traveling front,
\item  the decay estimate from Proposition \ref{difference} - a consequence of a Gaussian
kernel estimate,
\item  the local regularity results Proposition \ref{regularity} for non-degenerate parabolic equations and non-degenerate local solutions to the porous medium equations. The estimates there are more precise than the ones of the theorems.
\end{itemize}
Together these propositions allow to prove  the $\delta$-improvement of Proposition
\ref{improves}.
 From there, standard and easier arguments yield
$C^{1,\alpha}$ regularity of the boundary of the support, and the
pressure at the boundary, and a lower bound on the velocity $\nabla \rho^{m-1}$ at
the free boundary. After that, full regularity follows from Proposition \ref{regularity}.

\medskip


The first auxiliary result  we have mentioned deals with solutions of the PME, equation \eqref{PME}, that are global in space on a time interval $ (t_1,t_2) \subset \R $ with finite $ t_1<t_2$.

\begin{proposition} \label{prop_global}
 There exists $ \mu > 0 $ such that we have: \\
	If $ \rho_0: \Rn \to [0,\infty) $ satisfies
\begin{equation} \label{mucond}
\sup_{\Pos(\rho_0)} \left|\gradx \left( \rho^{m-1}_0-(x_n +t_1)\right) \right| \leq \mu
\end{equation}
	and if  $ \rho $ is the solution of $ (PME) $ on $ (t_1,t_2) \times \Rn $ with initial data $ \rho_0 $
at $t=t_1$, then for  $ k \in \N_0 $ and $ \alpha \in \N_0^n $
	\begin{equation} \label{Lipschitz}
		\sup_{\mathcal{P}(\rho)} (t-t_1)^{k+|\alpha|} \left|\dt^k \dx^\alpha \gradx \left( \rho^{m-1}-(x_n+t) \right)\right| \le C_1 \sup_{\mathcal{P}(\rho_0)} \left|\gradx \left( \rho^{m-1}_0-(x_n +t_1)\right) \right|
	\end{equation}
with $ C_1 = C_1(n,m,k,\alpha) > 0 $.
\end{proposition}

\noindent

 Since the equation is invariant under time translations, this is equivalent to restricting to  $t_1=0$. The statement is the main result of Kienzler \cite{Kienz14}. We announce that we will
take $ \mu $ equal to the constant of Proposition \ref{prop_global}
throughout the whole paper.

The next statement gives a pointwise control of differences of the
graphs of the pressures of two solutions assuming uniform control on
the gradients of the pressure. After a change of coordinates in
dependent and independent variables, this is a consequence of Gaussian
estimates of the second author \cite{Koc99}.

\begin{proposition} \label{difference}
Let $\rho$ and $\tilde \rho$ be global solutions to the PME on $ (0,\infty) \times \Rn $ with initial data $\rho_{0}$ and $\tilde \rho_0$ which satisfy \eqref{mucond}.

 {\rm (i)} Suppose that
\[
	|\rho^{m-1}_{0}-\tilde \rho^{m-1}_{0}| \le \delta \text{ for all } x \in \Rn\,,
\]
and  let $R \ge \delta $, $ 0 < t \leq R $ and $ x\in \R^n$. Then there exists $ c = c(n,m) $ such that for
\[
	a \ldef  c\Big\{ \delta e^{-\frac{R^2}{C(R+\rho_0^{m-1}(x))t}} +
   R^{-n} (R+\rho_0^{m-1}(x))^{\frac{m-2}{m-1}}
\int_{B_{R}(x)} |\rho_{0}(y)-\tilde \rho_{0}(y)| dy
+ \delta \left( \frac{\delta}{t}\right)^{\frac1{m-1}} \Big\}
\]
we have the pointwise comparison
\begin{equation} \label{comp.2.4}
\tilde \rho(t,x) \le \rho(t,x+a e_n) .
\end{equation}

 {\rm (ii)} If moreover
\[ \rho_0(y)= \tilde \rho_0(y) \quad \text{ for }  y \in B_R(x)\,,
\]
 then the conclusion \eqref{comp.2.4} holds with
\begin{equation} \label{aexp}
a= c\delta  \left(\frac{\delta}t\right)^{\frac1{m-1}}
e^{- \frac{R^2}{C(R+\rho^{m-1}_0(x)) t}  }.
\end{equation}
\end{proposition}
Note that the first part of Proposition \ref{difference} is only
useful for $ t \geq \delta $.  We also need a local regularity
statement under non-degeneracy conditions.

\begin{proposition} \label{regularity}
       	There exist $ \delta_5 > 0 $ and $\kappa_5 > 0 $ such that the following holds  if $ \rho$ is a $ \delta $-flat solution of $ (PME) $ on $(0,1) \times B_2(0) $ for a $ \delta < \delta_5 $:
\begin{itemize}
\item[(i)]  \label{deltaest}   We have	
\[
	\left|   \partial_t^k \partial^\alpha_x  \left(\rho^{m-1}(t,x)  -(x_n+t) \right) \right| \le C_5 \delta  \Big[ t^{-(k+\frac{|\alpha|}2)}(t+x_n)^{-\frac{|\alpha|}2} + (t+x_n)^{-k-\alpha} \Big]\,.
\]
if $ x_n+t \ge \kappa_5 \delta $ and $ (t,x) \in (0,1) \times B_1(0) $.

\medskip
\item[(ii)]  \label{SecondDerivatives}
If $(0,0)$ is contained in the free boundary, if
\[
\sup_{\Pos(\rho)\cap(-1,0)\times B_2(0)} \left|\gradx \left( \rho^{m-1}-(x_n +t_1)\right) \right| \leq \mu\,,
\]
then for   $ k \in \N_0 $ and $ \alpha \in \N_0^n $,
\[
  \sup_{((0,1] \times B_1(0))\cap \mathcal{P}(\rho)}  \left|\dt^k \dx^\alpha \left( \rho^{m-1}-(x_n+t) \right)\right| \le
C \delta
t^{\frac12-(k+\frac{|\alpha|}2)} (t+ \rho^{(m-1)})^{\frac12-\frac{|\alpha|}2}.
\]
\end{itemize}
 \end{proposition}

The first part is a purely parabolic estimate up to rescaling. The second part gives localized estimates of the type of Proposition \ref{prop_global}.


\section{Proof of Proposition \ref{improves}}\label{sec.prop1}

\begin{proof}

The proof is  organized in several  steps.

\subsection{Initial data for comparison solutions} We construct global
solutions much closer to $\rho$ than the traveling wave solutions.
Let $\rho$ be $\delta$-flat in $ Q_1 $ with $\delta \le \delta_0$
where $\delta_0$ will be chosen later on. In the first step we will
construct comparison functions $\rho_{\pm}$ as global solutions of the
PME on the strip $ (t_1,0) \times \Rn $ for some $t_1 \in
[-\frac{1}{2},0)$ by specifying initial data $ \rho_{\pm,0} $ at time
$t_1$ that satisfy the following properties:
\begin{equation} \label{globalTwBound}
\rho_{tw,-2\delta}(t_1,x) \le
  \rho_{-,0}(x) \le \rho_{+,0}(x) \le \rho_{tw,2\delta}(t_1,x) \text{
    for all } x \in \Rn,
\end{equation}
\begin{equation} \label{twFarOutside}
	\rho_{\pm,0}(x) = \rho_{tw,\pm2\delta}(t_1,x) \text{ for all } x \text{ with } |x| \geq \frac{3}{4},
\end{equation}
\begin{equation} \label{sandwichRhoNearZero}
	\rho_{-,0}(x) \le \rho(t_1,x) \le \rho_{+,0}(x) \text{ for all } x \text{ with } |x|\le 1,
\end{equation}
\begin{equation} \label{identification}
  \rho_{-,0}(x) = \rho(t_1,x) = \rho_{+,0}(x) \qquad \text{ for all } x \text{ with }
|x|\le \frac12 \text{ with} \ \rho^{m-1}(t_1,x) \ge \frac{4\delta}{\mu},
\end{equation}
and
\begin{equation} \label{initialGradEst}
	|\nabla \rho_{\pm,0}^{m-1}(x)- e_n | \le \mu \text{ for all } x \in \Pos(\rho_{\pm,0})
\end{equation}
where $ \mu $ is the constant from Proposition \ref{prop_global} and can be chosen as small as needed without loss of generality.

We postpone the construction of $\rho_{\pm,0}$ and the verification of its stated properties and explore first the consequences for the proof of the proposition. In this process we will choose
$t_1$ and $\delta_0$.

\subsection{The comparison solutions}
The gradient estimate \eqref{initialGradEst} and standard theory of
the PME ensure that there exist unique solutions $ \rho_\pm $ of the
PME on $ (t_1,0) \times \Rn $ with initial data $\rho_{\pm}(t_1,x)=
\rho_{\pm,0}(x) $.  By the maximum principle, the ordering
\eqref{globalTwBound} is preserved for these solutions for all
time. By Proposition \ref{prop_global} the solutions $\rho_{\pm}$ are
smooth in the sense of the proposition and satisfy the estimate
\eqref{Lipschitz} and the estimates  of Proposition \ref{regularity}.

Due to the identity  condition \eqref{twFarOutside} we can apply the second part of Proposition \ref{difference} to $\rho_{tw,-2\delta}(\tilde t+t_1,.)  $ and $\rho_-(\tilde t+t_1,.) $  resp. to $\rho_+(\tilde t+t_1,.) $ and $\rho_{tw,2\delta}(\tilde t+t_1,.) $ with
$R=\frac14$ and $ 0\le \tilde t \le |t_1|$ to get
\[
	\rho_-(t,x) \le  \rho_{tw,-2\delta} (t,x+ae_n)  \le    \rho_{tw,-\delta}(t,x)
\]
for $ t_1 \le t \le 0$ and $ |x| = 1 $ provided (see Proposition \ref{difference}, \eqref{aexp})
\begin{equation}
	\frac{a}{\delta} = c \left(\frac{\delta}{t-t_1}\right)^\frac{1}{m-1} e^{-\frac1{16c|t-t_1|}}
	\le c_2  e^{-\frac1{32c|t_1|}}
	\le 1
\end{equation}
for $ c_2 = c_2(n,m) $ which is the first restriction on $t_1$. It holds provided $|t_1|$ is
bounded by a constant depending only on $m$ and $n$.  Likewise, an application of Proposition \ref{difference} with $ R = \frac{1}{4} $ yields
$\rho_+(t,x) \ge \rho_{tw,+\delta}(t,x)$ on the same boundary set. Under this restriction we then have
\begin{equation} \label{boundary}
	\rho_-(t,x) \le \rho_{tw,-\delta}(t,x) \le \rho(t,x) \le \rho_{tw,\delta}(t,x) \le \rho_+(t,x)
\end{equation}
for $|x|=1$, and $t_1 \le t \le 0$ by the flatness of $ \rho $. Once we have this information at the boundary
we can apply the comparison principle and deduce from \eqref{boundary} and \eqref{sandwichRhoNearZero} that
\begin{equation}
	\rho_-(t,x) \le \rho(t,x) \le \rho_+(t,x)
\end{equation}
for $t_1 \le t \le 0 $ and $|x|\le 1$.

\subsection{The distance of the comparison solutions }
 Next, we apply Proposition  \ref{difference} with  $|x| \le \frac14$, $R=\frac14$  to $\rho_-(\tilde t+t_1,.)$ and $\rho_+(\tilde t+t_1,.)$
and $0< \tilde t\le |t_1|$. By construction  (more precisely condition \eqref{identification})
\[
\begin{split}
 \int_{B_{\frac14}(x)}  |\rho_{-,0}(y)-\rho_{+,0}(y)| dy
= &   \int_{B_{\frac14}(x)\cap \{ \rho^{m-1}(t_1,y) \le \frac{4\delta}{\mu}\} }
   |\rho_{-,0}(y)-\rho_{+,0}(y)| dy
\\  \le &     c_n \left( \frac{ \delta}{\mu}\right)^{\frac{m}{m-1}}
 \end{split}
\]
and hence $a$ in the first part of Proposition \ref{difference} is bounded by
\[
a\le c \delta \left[ e^{-\frac{1}{c(t-t_1)}} + \left(\frac{\delta}{t-t_1}\right)^{\frac1{m-1}} + \frac1{\mu}  \left( \frac{ \delta}{\mu}\right)^{\frac{1}{m-1}} \right].
\]
The exponential term is bounded by
\[
 e^{-\frac{1}{c(t-t_1)}} \le \delta (t-t_1)^4 \le \delta t_1^4
 \]
if $|t_1|$ is sufficiently small. If moreover $|t_1| \le \mu$ (recall that $\mu$ is the constant depending only on $n$ and $m$ in Proposition \ref{prop_global})
we obtain
\[
\rho_+(t,x) \le  \rho_-\left(t,x+  c_{n,m} \Big[ \Big(\frac{\delta}{t-t_1}\Big)^{\frac1{m-1}} + e^{-\frac{1}{C_5 (t-t_1)}}\Big]     \delta e_n\right),
\]
with a constant $ c = c(n,m) $. If $|x|\le \frac14$, $|t_1| \ll \frac12 $ (which we used above)
and
$\frac{t_1}2 \le t \le 0$ (to replace $t-t_1$ by $t_1$)
we obtain
\begin{equation}\label{sandw.2}
\rho_-(t,x) \le \rho(t,x) \le \rho_+(t,x) \le \rho_-\left(t,x +  \tau  \delta e_n \right), \quad  \tau:=  c_{n,m} \Big(\frac{\delta}{|t_1|} \Big)^{\frac1{m-1}} + |t_1|^4.
\end{equation}
We will complete the proof by studying the distance of the two solutions to the first order Taylor expansion of $\rho_-^{m-1}$ at a boundary point.

\subsection{Improved flatness}
 The pressure formulation of the porous medium equation shows that the
affine part of the Taylor expansion of solutions at the free boundary defines traveling wave solutions  at the
level of linear approximations: If $\rho$ is a solution with smooth pressure
$v= \rho^{m-1}$ and $\nabla_x v(t_0,x_0)  \ne 0 $ at the point $(t_0,x_0)$ of the boundary of the support then
\[ \dt v(t_0,x_0)   = |\nabla v(t_0,x_0)|^2 \]
and hence
\[
\left[(t-t_0) \dt v + (x-x_0) \nabla v(t_0,x_0)\right]_+ =
\rho_{tw}^{m - 1}(t-t_0,x-x_0;V,1, {\mathbf a})\]
with $ V = |\gradx
v(t_0,x_0)| $ and $ {\bf a}=V^{-1} \gradx v(t_0,x_0) $.  We use this
on $ \rho_- $, which has smooth pressure by Proposition
\ref{regularity}, at the point $ (t_0,x_0) = (0,h e_n) $, where $h$ is
the unique non-negative number for which $he_n$ is in the boundary of
the support of $\rho_-(0,\cdot) $, hence
\[
	V= |\nabla \rho_-^{m-1}(0,he_n)|, \qquad \mathbf{a} =  V^{-1}  \nabla \rho_-^{m-1}(0,he_n).
\]

Higher derivatives of $\rho_{\pm}$ in $(t_1,0) \times B_{\frac14}(0)$ are
controlled by Part (ii)  of Proposition
\ref{regularity}, which is applicable because of
\eqref{initialGradEst}, hence in the positivity set
\[  |t-t_1|^{\frac{k + |\alpha|}2}(t+ \rho^{m-1}_{\pm}(t,x))^{\frac{|\alpha|}2}   \left|
\partial_t^k \partial^\alpha_x  (\rho^{m-1}_{\pm}- (x_n+t)) \right|\le C_2 \, \delta.
\]
In particular the remainder term of Taylor's formula applied in a ball with radius $ r $ centered at $ (0,h e_n) $ can be bounded
in the positivity set of $ \rho_- $ by
\[ |\rho^{m-1}_-(t,x)- \rho_{tw}^{m-1}(t,x;V,1,\mathbf{a})| \le   c (r/|t_1|)^{2}\delta \]
with $\tau$ from \eqref{sandw.2} provided
\[ |t| \le 2r \le |t_1|, \quad |x-he_n| \le 4r.  \]
Using this estimate  and \eqref{sandw.2}, by Taylor's formula the  graph of $\rho^{m-1}$ is sandwiched as
\begin{equation}\label{sandwich}    \rho_{tw}(t,x- c(r/|t_1|)^2\delta e_n;V,1,\mathbf{a})
\le   \rho(t,x) \le  \rho_{tw}\left(t,x+[c(r/|t_1|)^2 +\tau]  \delta e_n; V,1, \mathbf{a} \right).
\end{equation}
We define
\[ \tilde \rho(t,x)  = r^{-\frac1{m-1}} \rho(r t,r x) \]
which is sandwiched as
\begin{equation}\label{sandwich3}
   \rho_{tw}(t,x- cr|t_1|^{-2} \delta e_n;V,1,\mathbf{a})
\le  \tilde \rho(t,x)
 \le  \rho_{tw}\left(t,x+[cr|t_1|^{-2}  + \tau/t ]  \delta e_n; V,1,\mathbf{a}\right).
\end{equation}
We choose $t_1 $ so small that the previous conditions are satisfied, and in
addition $(c+1)|t_1| < \frac1{40}$, next $r= |t_1|^3$ and finally
$ \delta_0 \le \mu r^{m+2}/40 $
which ensures $\tilde \rho$ is sandwiched as
\[
	\rho_{tw}(t,x- \delta e_n/10;V,1,\mathbf{a}) \leq \tilde \rho(t,x) \leq \rho_{tw}(t,x+\delta e_n/10;V,1,\mathbf{a})
\]
for $|x|\le 2$ and $-2\le t \le 0$.

\subsection{Rescaling}
 We now compare the two approximations we constructed  for $\rho$ near $(0,0)$ after the rescaling with parameter $r$ to estimate $\lambda$ and $\nu$ of the statement of the Proposition. On the one hand,  $\delta$-flatness in the direction $e_n$ at $t =0$ implies that
\[
(x_n-\delta/r)_+ \le \tilde \rho^{m-1} (0,x)  \le (x_n +\delta/r)_+  \]
for $|x|\le 1$. On the other hand, we have
\[
\Big( V^{-1}  \mathbf{a} \cdot (x- \frac{1}{10} \delta  e_n \Big)_+  \le \tilde \rho^{m-1}(0,x) \le \Big( V^{-1} \mathbf{a} \cdot (x-\delta  e_n \Big)_+ \]
 for $|x|\le 2$.   We evaluate the inequalities at $e_n$ and at $\mathbf{a}$:
\[
\begin{split}
1 -\frac{\delta}r \le  c (\mathbf{a}_n  + \delta/10 ), & \quad
c  (\mathbf{a}_n - \delta/10) \le   1+ \frac{\delta}r, \\
 \mathbf{a}_n -\frac{\delta}r \le  c ( 1+  \mathbf{a}_n\delta/10),& \quad
\frac1{\lambda} (1-h\mathbf{a}_n\delta/10) \le \mathbf{a}_n + \frac{\delta}r\,.
\end{split}
\]
Thus, if $\delta_0 \le r/10$ - which we can satisfy - then
\[
V^{-1}  \le \frac{1+ \delta}{1-\frac{\delta}r } \le 1+ 2 \frac{\delta}r, \quad
V^{-1} \ge \frac{1-\delta}{1+\frac{\delta}r} \ge 1- 2 \frac{\delta}r,
\]
\[
\mathbf{a}_n \ge  V^{-1}  (1-\delta/r)-\frac{\delta}10  \ge 1-4\frac{\delta}r,
\quad
 |V'|^2 =1- V_n^2  \le 8     \frac{\delta}r.
 \]
 Thus  $ r/V < 1$ and the proof is complete, up to a construction of the initial data $\rho_{\pm,0}(x)$ .

\subsection{Construction of the comparison functions $\rho_{\pm,0}$}   We define
\[
	\rho_{\pm, R}^{m-1}(x)\ldef \rho_{tw,\pm2\delta}^{m-1}(t_1,x) \mp  12 \delta  (\frac{3}{4} - |x|)_+
\]
and
\[
	\rho_{\pm, L }^{m-1}(x)\ldef [\rho_{tw,\pm 2\delta}^{m-1}(t_1,x) \mp   \mu  (x_n+t_1\pm 2\delta)]_+,
\]
and with these notations
\[ \rho_{-,0}(x) \ldef
	\begin{cases}
		\min \{\rho_{-,R}(x), \; \rho_{-,L}(x), \; \rho(t_1,x) \} &\text{ if } |x| \leq 1 \\
		\rho_{tw,-2\delta}(t_1,x) &\text{ if } |x| > 1
	\end{cases}
\]
and
\[ \rho_{+,0}(x) \ldef
	\begin{cases}
		\max \{\rho_{+,R}(x), \; \rho_{+,L}(x), \; \rho(t_1,x) \} &\text{ if } |x| \leq 1 \\
		\rho_{tw,2\delta}(t_1,x) &\text{ if } |x| > 1.
	\end{cases}
\]
With this, \eqref{globalTwBound}, \eqref{twFarOutside} and \eqref{sandwichRhoNearZero} are obvious.
But the definition also implies that
\begin{equation}
	\rho_{-,0}(x)  = \rho(t_1,x) = \rho_{+,0}(x)  \qquad \text{ if } |x|\le \frac12, \text{ and } t_1+x_n \ge  3\max\{ 2\delta,  \frac{\mu}{C_5 \delta} \}\,,
\end{equation}
 hence \eqref{identification}.

We set $t_1 \ge -\frac12$. For $ \Lambda>\max\{\kappa_5 \delta_0, 4\delta \} $ (which we will ensure later on) we use the first part of Proposition \ref{regularity} to obtain
\[
	|\gradx \rho_{\pm,0}^{m-1}(x)-e_n | \le C _5\delta (x_n+t_1)^{-1}
\]
for $|x|\le \frac12 $ and $-\frac{1}{2} \le  t_1$ and $x_n+t_1\ge \Lambda $.
We will choose
\begin{equation*}
	\Lambda =  \frac{4\delta}{\mu}
\end{equation*}
assuming without loss of generality that $\mu \le \min\{\kappa_5^{-1},\frac1{16}\}  $ and $\delta \le \delta_0 \le \mu$.

For $ |x|\ge \frac12 $ or $t_1+x_n  <\Lambda$
 the same bound follows directly from the definition of $ \rho_{\pm,0} $.  Altogether, we then have
\begin{equation*}
   \left| \nabla_x (\rho_{\pm,0}^{m-1}(x) - (x_n+t_1) ) \right| \le \mu \text{ for } x \in \Pos(\rho_{\pm,0})
\end{equation*}
and thus \eqref{initialGradEst} holds.
\end{proof}


\section{Proof of Theorem \ref{th-main}}\label{proofthm}

In this section we deduce $C^{1,\alpha}$ regularity claimed in Theorem
\ref{th-main} from the Proposition \ref{improves} of the previous
section.  Full regularity follows from Proposition \ref{regularity}.
The iteration argument using the improvement of flatness of
Proposition \ref{improves} goes back to Caffarelli's work, see
\cite{CW90} for example.

\begin{proof}

We first prove the claim:

\begin{claim} \label{claim}
There exists $ \alpha > 0 $ and $ \delta > 0 $ such that the following holds:

If the solution $ \rho $ is $ \delta_1 $-flat on $ Q_1=(-1,0) \times
B_1(0) $ for a $ \delta_1 < \delta $ with speed 1 and $ {\bf a }= e_n
$, then
\[
\Vert \gradx \rho^{m-1}   - e_n \Vert_{L^\infty\big(([-1/2,0] \times  B_{1/2}(0)) \cap   \mathcal{P}(\rho)\big)} < c \, \delta.
 \]
 Moreover, $ \rho^{m-1} \in C^{1,\alpha}(((-1/2,0] \times B_{1/2}(0)) \cap
   \mathcal{P}(\rho))$ in the sense that the derivatives are H\"older continuous up to the boundary.
\end{claim}
The claim  implies that the free boundary is the graph of a function $h \in
C^{1,\alpha}$.

\subsection{Proof of the claim: Setup}
By assumption, $\rho$ is a $\delta$-flat solution with $\delta \le \delta_1:= \delta_0 t_0/2$, with $\delta_0$ and $t_0$ as in Proposition 1.    Let us change the origin to a point $(s,y)$ of the free boundary with $-\frac12 \le s \le 0 $ and $|y|\le \frac12$. If $\delta \le \delta_1$ and if $\rho$ is $\delta$-flat then  the function
\[
\rho_0(t,x)   = 2^{\frac1{m-1}} \rho(\frac12 t+s, \frac12 x+y)
\]
is $2\delta$-flat, and we can apply Proposition \ref{improves} because of our assumption on the smallness of $\delta$. Thus, there exists $V$ and a unit normal vector $\mathbf{a}$ so that the newly rescaled function
\[
\rho_1(t,x) =   (2/t_0)^{\frac1{m-1}} \rho\Big( \frac{t_0}{2V^2}t+ s , \frac{t_0}{2V}  x+y\Big)
\]
satisfies
\[
 (t+ \langle \mathbf{a}, x\rangle - \delta)_+ \le
 \rho_1(t,x)^{m-1}   \le    (t+ \langle \mathbf{a} , x\rangle + \delta)_+
 \]
i.e., $  \rho_1(t,x)$ is $\delta$-flat in direction $\mathbf{a}$ with velocity $V$, and
$2V\delta/t_0$ flat in direction $e_n$ with velocity $1$, hence
\[
|\mathbf{a} - e_n| + |V-1| \le 2\delta/t_0.
\]


\subsection{Proof of the claim: Iteration}
We repeat this construction recursively, but now we keep the point
$(0,0)$ fixed since it is in the free boundary.  Then there exist a
sequence of unit vectors $\mathbf{a}_j$ and numbers $W_j$ so that
\[
\rho_j(t,x) =  r^{-\frac{j}{m-1}}    \rho\Big(W_j^{-2} r^j t+s , W_j^{-1} r  x+y \Big)
\]
is $2^{1-j}\delta $-flat in direction  $\mathbf{a}_j$ at every step.
More precisely
\[
 \Big(t+ \langle  \mathbf{a}_j, x\rangle - 2^{1-j} \delta\Big)_+^{\frac1{m-1}} \le
 \rho_j(t,x) \le   (t+ \langle \nu_j, x\rangle + 2^{1-j} \delta)_+^{\frac1{m-1}} .
\]
 Moreover,
\[
\left| \frac{W_{j+1}}{W_j} -1\right|= |V_j -1|  \le  2^{1-j} \delta
 \]
\[  \left|  \mathbf{a}_{j+1} -\mathbf{a}_j  \right| \le  2^{1-j}\delta.
\]
since $\rho_j$ is $(\frac12)^{j-1}\delta$ flat.
Summing  a geometric series,
\[
|\mathbf{a}_j - e_n | + |\Lambda_j-1| \le c \delta\,,
\]
and both quantities are Cauchy sequences with limit $\mathbf{a}(s,y)$  and
$W(s,y)$ which are functions of the point considered initially  and which satisfy
\[ |W(s,y)-1|  + |\mathbf{a}(s,y)  -1|\le  c\delta.\]
 Without loss of generality  we may assume that $W(s,y) = 1$ at the expense of increasing
$\delta$ by a fixed factor, and that $\mathbf {a}(s,y)= e_n$ for a fixed point $(s,y)$ in the free boundary.

\subsection{Proof of Theorem \ref{th-main}: $C^{1,\alpha}$ regularity at the free boundary}  We claim that this analysis shows that there exists $C>0$ so that
\begin{equation}\label{form.alpha}
(t + \langle x,\mathbf{a}(s,y) \rangle -CR^{\alpha_1})_+^{\frac1{m-1}}
\le   \rho(s+ \Lambda^2(s,y)t,y+\Lambda x ) \le    (t + \langle x,\mathbf{a}) +CR^{\alpha_1})_+^{\frac1{m-1}}
\end{equation}
if $|x-y|+|t-s| \le R/C $, $\nu$ depends on the free boundary  point $(s,y)$ and
\[ \alpha_1 = \frac{\ln 1/2}{\ln r}>1.  \]
  This is a rather standard counting argument, but we will give the details for the reader's convenience. Place yourself at one such point $(t,x)$ and count the maximal number of iterations $N$ so that the initial unit cylinder of the definition of $\delta$-flatness
for $\rho$ is shrunk but still contains $(t,x)$. Looking at the scalings of Proposition 1, this means for $R$ small (so that $N$ is large) we have roughly (but it will be precise enough)
\[
r^N \sim R \,.
\]
But the final flatness after $N$ steps is $\delta_N=\delta 2^{1-N}$, hence
\[
\delta_N\sim \delta \, 2^{-\log (R \Lambda^{-2})/ \log r}=C\delta R^{-\log 2/ \log r}=
C\delta R^{\log (1/2)/\log r}
\]
and this immediately implies  \eqref{form.alpha}.

\subsection{Proof of the Theorem \ref{th-main}: $C^{1,\alpha}$ regularity of the pressure}
 Equation \eqref{form.alpha} implies one sided
differentiability of $\rho^{m-1}$ at the free boundary, and H\"older continuity
of the derivative at the free boundary.

For $(t,x)$ and $(s,y)$ in the positivity set we want to prove
\[ |\nabla \rho^{m-1}(t,x) - \nabla \rho^{m-1}(s,y)| \le c (|t-s|+|x-y|)^{\alpha}. \]
By the triangle inequality it suffices to consider two cases.
If  $(t_0,x_0)$ is at the free boundary, and $(t_1,x_1)$ is not we
denote by  $d$ be the distance to the free boundary. By \eqref{form.alpha}
the rescaled function
\[  \tilde \rho(t,x) = (d/2)^{-\frac1{m-1}}  \rho(dt/2,dx/2)  \]
 is $d^\alpha \delta$ flat. By Proposition
\ref{regularity} the higher  order derivative derivatives of $\tilde \rho^{m-1}$ are uniformly bounded
 in a ball of radius $1/4$ around $(dt_1/2,dx_1/2)$. This is only compatible with $d^\alpha \delta$ flatness if the second order derivatives in the ball of radius $1/4$ are bounded by a constant times $d^\alpha \delta$. Then the same is true for first order derivatives, and the derivatives are $d^\alpha \delta$ close to the corresponding derivatives of the powers of
 the traveling wave solutions, and hence also to $D_{t,x} \rho^{m-1}(t_0,x_0)$.

 In the second case both $(t,x)$ and $(s,y)$ have the same distance to
 the free boundary and to another. After scaling this reduces to
 estimating second derivatives at points of distance $1$ to the free
 boundary in terms of $\delta$. This is the contents of Proposition
 \ref{regularity}, which also implies full regularity.
\end{proof}

\section{Von Mises transform and intrinsic geometry}\label{sec.vonmises}

For the proofs of Proposition \ref{difference} and \ref{regularity} we
want to linearize simultaneously the geometry and the porous medium
equation. For that purpose we change coordinates with respect to
dependent and independent variables simultaneously. After this change
we obtain an intrinsic subelliptic degenerate parabolic equation.
This section is devoted to this change of coordinates and a discussion
of the sub-Riemannian geometry associated to linearization of the
traveling wave.

Consider the von-Mises-transform of $ \rho^{m-1} $ denoted by $ w $,
assuming that the pressure $\rho^{m-1}$ is Lipschitz continuous and
$\partial_{x_n} \rho^{m-1}$ is bounded from below by a positive
constant on the positivity set.  We decompose $\R^n=\R^{n-1}\times \R
$ and write $ x = (x',x_n)$. We define the bi-Lipschitz map
\[
(t,x) \to (t,x',\rho^{m-1}(t,x))=(t,y)
\]
from the closure of the positivity set to $\{(t,y): y_n\ge 0\}$.  We
now introduce the inverse map: $w(t,y) = x_n$, and we will use
$w(t,y)$ instead of $\rho(t,x)$ as the basic unknown in our
computations. A tedious but standard calculation gives
\[
\partial_t w - y_n \Delta_x' w + y_n^{-\sigma} \partial_{y_n}
\Big[y_n^{1 + \sigma} \frac{1 + |\nabla_y' w|^2}{\partial_{y_n} w}
\Big]= 0
\] with
\begin{equation}
\sigma \ldef  \frac{2-m}{m-1} > -1.
\end{equation}
Here $\Delta'$ and $\nabla'$ denote the operators with respect to the
first $(n-1)$ space variables.  In this setting, the traveling wave
solution described in the introduction becomes
\[
w(t,x) = x_n -(1+\sigma) t\,,
\]
and the deviation from the traveling wave solution
 \[
 u = w -(x_n -(1+\sigma) t)
 \]
 satisfies the equation
 \begin{equation} \label{TPE}
\ds u - y_n \, \lapy' u - y_n^{-\sigma}
   \, \dn \Big[ y_n^{1+\sigma} \, \frac{\dn u -|\grady'u|^2}{1+ \dn u
   }\Big] = 0 \text{ on } (s_1,s_2) \times \overline{H},
\end{equation}
where $ (s_1,s_2) $ is a suitable rescaling of the original time
interval $ (t_1,t_2) $ (see \cite{Kienz14}). This equation can be
rewritten as a quasilinear equation
\begin{equation} \label{TPE_coeff}
	\ds u - y_n \, a^{ij}(Du)  \dij u - (1 + \sigma) \, b^j \, \dj u = 0
\end{equation}
with symmetric coefficients $ a^{ij} $ given by
\[ a^{ij} = a^{ij}[\grady u] =
	\begin{cases}
          \delta^{ij} & \text{ for } i,  \, j < n \\[0.2cm]
          - \frac{\di u}{1 + \dn u} & \text{ for } i < n, \, j = n \text{ or } i=n, j<n \\[0.2cm]
		\frac{1 + |\grady' u|^2}{(1 + \dn u)^2} & \text{ for } i = j = n,
	\end{cases}
\]
and
\[ b^j = b^j[\grady u] =
	\begin{cases}
		a^{nj} & \text{ for } j < n \\[0.2cm]
		\frac{1}{1 + \dn u} & \text{ for } j = n.
	\end{cases}
\]
We may also write the equation as a perturbation of the linear equation
\begin{equation} \label{TPE_inhom}
	\ds u - L_\sigma u = f[u]
\end{equation}
with inhomogeneity
\[ f[u] \ldef - y_n^{-\sigma} \dn \left( y_n^{1 + \sigma} \, \frac{|\grady u|^2}{1 + \dn u} \right) \]
and spatial linear operator
\[ L_\sigma u \ldef y_n \, \lap_y u + (1 + \sigma) \, \dn u. \]

The second order part of $L_\sigma$ defines a Riemannian metric $g$ on
$y_n>0$,
\[ g(x) (v,w) =   x_n^{-1} v\cdot w \]
which in turn defines a nonstandard metric on the closed upper half space
which is  the Carnot-Carath\'eodory metric $ \dist $ defined by the vector fields
\[  x_n^{1/2}  e_j \]
(see \cite{Kienz14}).

We denote intrinsic balls by $B^i_r(x)$. They
are related to Euclidean balls as follows.
Given arbitrary $ r > 0 $, $ y_0 \in \overline{H} $ and using the abbreviation
\[ R \ldef r \, (r + \sqrt{\yzeron}), \]
\[ B_{R/C} (x) \subset B_{\delta_1 r}^i(x) \subset B_{cR}(x) \]
for some $c>1$.

\section{Proof of Proposition \ref{difference}}
For the proof of this proposition we need a general Gaussian estimate in terms of the intrinsic metric and weighted measure, as already contained in \cite{Koc99}. We  provide a new and simpler proof with a standard strategy  as follows.

\begin{lemma} \label{Gaussian}
Let $ \sigma>-1$ and $a^{ij}(t,x)$ measurable,  uniformly bounded and coercive,
\begin{equation}
a^{ij}(t,x) \xi_i \xi_j \ge \nu |\xi|^2
\end{equation}
for almost all $t$ and $x$. We consider the equation
\begin{equation}\label{new.eq}
u_t - x_n^{-\sigma} \partial_i \left(x_n^{1+\sigma} a^{ij} \partial_j u \right) = 0 \,.
\end{equation}
Then there is a unique Green's function $g(t,x,s,y)$ such that the unique solution to the initial value problem for \eqref{new.eq}  is given by
\[ u(t,x) = \int g(t,s,x,y)      u_0(y) dy.
\]
 Moreover, there exist $c$ and $\delta>0$ so that we have the estimate
 \[
 |g(t,x,s,y)| \le c y_n^\sigma    | (x_n+y_n+ |t-s|)^{\frac12} |t-s|^{\frac12}|^{-n-\sigma}       e^{- \delta \frac{ |x-y|^2}{ (x_n+y_n +|x-y|) (t-s) } }.
\]
\end{lemma}
\begin{proof}
The formal energy identity
\[
\frac12 \int y_n^\sigma |u(t,x)|^2 dx -\frac12 \int y_n^\sigma |u(s,x)|^2 dx
+ \int_s^t\int y_n^{1+\sigma} a^{ij} \partial_i u \partial_j u dx d\tau = 0
\]
can be used with a Galerkin approximation to construct a weak solution for given initial data which satisfies this energy identity. Let $\phi$ be a bounded
Lipschitz function
\begin{equation}
x_n |\nabla \phi|^2 \le 1.
\label{lipschitz}
\end{equation}
Note that this condition is equivalent to the requirement that $\eta$ is a Lipschitz function with Lipschitz constant $1$ with respect to the special Riemannian metric adapted to the problem.
 Then, again formally, but with standard justification, we have the weighted energy estimate
\begin{equation}
 \int e^{L\phi(x,y)}y_n^\sigma |u(t,x)|^2 dx
\le e^{c L^2 (t-s)} \int y_n^\sigma e^{L\phi(x,y)}|u(s,x)|^2 dx.
\end{equation}
Such estimates are called Davies-Gaffney estimates, see \cite{Da1}.

The next ingredient is the Moser iteration.
\begin{lemma}\label{moser}  Let  $B^i_{2R}(x)$ be the intrinsic ball. We assume that $u$ is a weak solution on $Q_{2R} = (-(2R)^2, 0] \times B_{2R}(x)$.  Then
\begin{equation}  \Vert u \Vert_{L^\infty(Q_R) }
\le c |Q_R|_\sigma^{-1}    \Vert u \Vert_{L^2(Q_{2R},x_n^s dx)} .
\end{equation}
\end{lemma}

\begin{proof}
Let $k >1$. We define the cylinders $ Q_j = Q_{(1+2^{-j}R) }$
and Lipschitz functions $\eta_j$ with $\eta_j= 1$ on $Q_j$ and $\eta_j=0$
on $Q \backslash Q_{j-1}$,
\[
|\partial_t \eta_j| , \quad x_n^{\frac12}  |D_x \eta_j| \lesssim  2^j \,.
\]

The starting points for Nash's inequality \eqref{nash}  are the Sobolev inequality
\[
\Vert u \Vert_{L^{\frac{n}{n-1}} (\mathbb{R}^n)} \le 2 \Vert \nabla  u \Vert_{L^1(\mathbb{R}^n)}
\]
and Hardy's inequality for $s> -1/p$,
\[  \Vert x_n^s u \Vert_{L^p(H)} \le  \frac1{s+\frac1p}    \Vert x_n^{s+1} \nabla u \Vert_{L^p(H) }\,.
\]
By an even reflection the Sobolev inequality holds in $H$,
\[
\Vert u \Vert_{L^{\frac{n}{n-1}} (H)} \le 2 \Vert \nabla  u \Vert_{L^1(H)}
\]
and hence, for $ s\ge 0 $,  by an  application of H\"older's inequality and Sobolev's inequality
\[
\begin{split}
 \Vert x_n^s u \Vert_{L^{\frac{n}{n-1}}}^{\frac{n}{n-1}} = &
  \int_0^\infty  (x_n^s \Vert u(.,x_n) \Vert_{L^1(\mathbb{R}^{n-1})})^{\frac1{n-1}}
  x_n^s \Vert u(.,x_n) \Vert_{L^{\frac{n-2}{n-2}}} dx_n
\\ \le  &  \sup_{x_n} (x_n^s \Vert u(.,x_n) \Vert_{L^1})^{\frac{1}{n-1}}
\Vert x_n^s \nabla u \Vert_{L^1}
\end{split}
\]
and
\[  t^s \Vert u(.,t) \Vert_{L^1}
\le t^s \Vert \partial_n u \Vert_{L^1(\{x_n >t \})}
\le \Vert x_n^s \partial_n u \Vert_{L^1}, \]
thus for $ s \ge 0$
\[
\Vert x_n^s u \Vert_{L^{\frac{n}{n-1}}(H) } \le 2 \Vert x_n^s \nabla u \Vert_{L^1(H)}
\]
We combine this inequality with the Hardy inequality for $p=1$ and we get
for $ 1 \le p \le \frac{n}{n-1}$
\[
 \Vert x_n^{s-\frac{n-1}{ p}  -n}    u \Vert_{L^p}
\le c_{n,p}  \Vert x_n^s \nabla u \Vert_{L^1}. \]

We apply this inequality to $ |u|^p$ and obtain for $ \frac1q-\frac1n \le \frac1p \le \frac1q $ the weighted Hardy-Sobolev inequality

\begin{equation}
\Vert x_n^{s-(\frac{n}p- \frac{n}q+1)}    u \Vert_{L^{p}(H) } \le c_{n,p} \Vert x_n^s \nabla u \Vert_{L^q(H)}\,,
\end{equation}
which in turn implies Nash's inequality for $s > -1$,
\begin{equation}\label{nash}
\Vert u \Vert_{L^{2+\frac2{n+s}}(x_n^s)}^{2+\frac2{n+s}}
\le c_n \Vert u \Vert_{L^2(x_n^s)}^{\frac2{n+s}} \Vert \nabla u \Vert_{L^2(x_n^{1+\sigma})}^2
\end{equation}

We formally calculate with $ \eta= \eta_j$, $p =p^j$ and $R=1$
\[
\begin{split}
\frac{d}{dt} \int x_n^\sigma \eta^2 u^{p}dx = & p \int_H x_n^\sigma
\eta^2 u^{p-1} u_t dx + 2 \int_H x_n^\sigma \eta \eta_t u^2 dx \\ = &
-p \int_H x_n^{1+\sigma} \partial_i u a^{ij} \partial_j (\eta^2
u^{p-1}) dx + 2 \int_H x_n^\sigma \eta \eta_t u^2 dx \\ = & -\int_H
\frac{4p}{(p-1)^2}x_n^{1+\sigma} \eta^2 a^{ij} \partial_i |u|^{p/2}
\partial_j |u|^{p/2} dx \\ & - 4 \int_H x_n^{1+\sigma} (\partial_i
u^{p/2}) a^{ij} \eta (\partial_j \eta) u^{p/2}) dx + 2 \int_H
x_n^\sigma \eta \eta_t u^p dx
\end{split}
\]
and hence, with $v = u^{p/2}$,
\[ \sup_t \int x_n^\sigma \eta^2  |v(t,x)|^2 dx
+ \frac{1}{p}
\int x_n^{1+\sigma} \eta^2 |\nabla v|^2 dx
\le   p 2^{2j} \Vert v \Vert^2_{L^2(Q_{j-1})}. \]
We combine this with Nash's inequality to
\[ \Vert u \Vert^{p^j}_{L^{p^j}( Q_j) } \le c  p^{2j} 2^{2j}  \Vert u \Vert_{L^{p^{j-1}}}^{p^j}
\]
and hence
\[ \Vert u \Vert_{L^{p^j}(Q_j)} \le ( c p^{2j} 2^{2j})^{\frac1{p^j}}  |Q|_{\sigma}^{\frac1{p^{j-1}} - \frac1{p_j}}
\Vert u \Vert_{L^{p^{j-1}}(Q_{j-1})} \]
which becomes
\[ \Vert u \Vert_{L^{p^k}( Q_k) }
\le \prod_{j=1}^k  ( c (2p)^{2j} )^{\frac1{p^j}}  |Q|_\sigma^{\frac{1}{p^k}-\frac12}  \Vert u \Vert_{L^2}. \]
Then
\[ c^{\sum\frac1{p^j}} \le  c^\Lambda \]
for some $\Lambda$ and
\[ p^{2\sum \frac{2j}{p^j}} \le P \]\,.
Now we let $j$ tend to $\infty$ to obtain  Moser's inequality \eqref{moser}.
\end{proof}

We continue with the proof of the Gaussian estimates of Lemma
\ref{Gaussian}.  In the next step we combine the
consequences of the Moser iteration and the Davies-Gaffney estimate,
and we obtain the Gaussian kernel bounds. This argument is general
and standard, and we only indicate the steps. First, let $x_1$ and
$x_2$ be two points with distance $R$. If $u_0\in L^2$ is supported in
an intrinsic ball $B_{R/4}^i(x_1)$ then by the Davies Gaffney estimate
\[
\Vert u(t) \Vert_{L^2(B_{R/4}^i)} \le c e^{t L^2   - LR/2  }
\Vert u_0 \Vert_{L^2}
\]
for which we optimize with $L= \frac{R}{4t} $  which yields
\[
\Vert u(t) \Vert_{L^2(B_{R/4}^i)} \le c e^{- \frac{R^2}{16t}}
\Vert u_0 \Vert_{L^2}.
\]
Now we use the Moser estimate \eqref{moser} to conclude that for  $x \in B_{R/4}^i(x_2)$
\begin{equation}\label{half}
 |u(t,x)| \le c |B_{R/4}^i(x_2)|^{-\frac12} e^{-\frac{R^2}{16t}}
\Vert u_0 \Vert_{L^2(B_{R/4}^i(x_1))}.
\end{equation}
By the Riesz representation theorem the map
\[ L^2\ni u_0 \to u(t,x) \]
has a kernel
\[ g(t,s,x,y) \]
which by duality satisfies
\[ \partial_s g + \partial_i a^{ij} \partial_{y_j} g = 0. \] Let
$v(s,y) = g(t,s,x,y)$.  Then estimate \eqref{half} combined with the
Riesz representation theorem reads as
\[ \Vert v(s) \Vert_{L^2(B_R(x_1)^i)} \le |B_{R/2}^i|^{-\frac12} e^{-\frac{R^2}{16(t-s)}} \] and a second application of Moser's estimate implies the Gaussian bounds.
\end{proof}

\begin{proof}[Proof of Proposition \ref{difference}]
  The proof consists in tracing the assumptions through the von Mises
  transform. The conclusion is a consequence of the pointwise Gaussian
  bound.  Let $ w_1 $ and $ w_2 $ be two solutions of the transformed
  equation.  For the difference $ w_2 - w_1 = w $ we have
\[ \partial_t w - x_n \Delta_x' w +
 x_n^{-\sigma} \partial_{x_n}\left[  x_n^{1 + \sigma} \Big( \frac{1 +
   |\nabla_x' w_2|^2}{\partial_{x_n}  w_2}
- \frac{1 +
   |\nabla_x'  w_1|^2}{\partial_{x_n} w_1}
 \Big)\right]  = 0.
 \]
We  consider this as a linear equation
\[ \partial_t \hat w - x_n
 \Delta_x' \hat w - x_n^{-\sigma} \partial_{x_n} \left[ x_n^{1 + \sigma}
 \left( \frac{1+ |\nabla' w_2|^2}{\partial_n w_1 \partial_n w_2} \partial_n w - \sum_{i=1}^{n-1}
\frac{\partial_i(w_2+w_1)  }{\partial_n w_2} \partial_i w\right)\right]   =0  \]
with coefficients
\[ a^{nj}  = -\frac{\partial_j(w_2+w_1)}{\partial_n w_2}, \qquad
a^{nn}  = \frac{1+|\nabla' w_2|^2}{\partial_n w_2 \partial_n \tilde w_1}, \qquad a^{ij} = \delta_{ij} \qquad \text{ if } i<n \]
satisfying
 \[ |a^{ij}-\delta^{ij}| \le c \delta \]\,.
For such an equation we have the Gaussian estimate from Lemma \ref{Gaussian}.
The representation of the  solution  yields
\[
|w(t,x)| \leq  c
 |B_{\sqrt{t}}(x)|_\sigma^{-1} \int_H y_n^\sigma e^{- C
   \frac{d^2(x,y)}{t} } |w_2(0,y)-w_1(0,y)| dy.
\]
The claim of the proposition  follows if we prove
\[
\begin{split}  |B_{\sqrt{t}}(x)|_\sigma^{-1} \int_H y_n^\sigma e^{- C
   \frac{d^2(x,y)}{t} } |\tilde w(0,y)-w(0,y)| dy
& \\ & \hspace{-7cm} \le c \left\{ \delta (e^{-\frac{r^2}{C(r+\rho_0^{m-1}(x))t}}+ (\delta/t)^{\frac1{m-1}}
+ r^{-n}(r^{\frac1{m-1}} + \rho_0(x))^{m-2} \int_{B_r(x)}|\rho_0(y)-\tilde \rho_0(y)| dy   . \right\}
\end{split}
\]
The assumptions imply \ $|\tilde w - w| \le 2 \delta $,
and we claim that with the intrinsic balls $B^i_r\subset H$
\[
|B_{\sqrt{t}}^i(x)|_{\sigma}^{-1} \int_{H \backslash B_R^i(x)}
y_n^\sigma e^{-C \frac{d^2(x,y)}t} dy \le  c e^{- C\frac{R^2}{t}}
\le c e^{-C \frac{r(r+\rho^{m-1}(\hat x))}{t}}
\]
which follows by a straight-forward calculation and the observation that
$  B_r(x) \subset B_R^i(x)$ provided $  R \ge  r(r+x_n)$.
To complete the proof we show that
\begin{equation}
\begin{split}
  |B_{\sqrt{t}}^i(x_0)|_\sigma^{-1}   \int_{H \cap  B^i_{\sqrt{t}}(x_0)}
y_n^\sigma |w_2 - w_1 | dx
& \\ & \hspace{-5cm} \le C \Big[ r^{-n} (r^{\frac1{m-1}} + \rho_0(y_0))^{m-2} \int_{B_r(y_0)} |\tilde \rho_0(y)-\rho_0(y)| dy  + \delta  (\delta/t)^{\frac1{m-1}} \Big]
\end{split}
\end{equation}
where $y_0$ and $x_0$ are related by the coordinate change, i.e. $y_0= (x_0',\rho^{m-1}(t_1,x_0))$ and
\[    r(r+ \rho^{m-1}(x_0)) \sim  \sqrt{t}.  \]
Then
\[
\begin{split}
\int_{H \cap B^i_{\sqrt{t}}(x_0)\cap\{x_n>8\delta\} }
y_n^\sigma |\tilde w - w | dx
\sim & \int_{B_r(x_0)\cap \{ \rho^{m-1}> 4\delta\}}
\rho^{2-m}(x) |\tilde \rho^{m-1}(x)-\rho^{m-1}(x) | dx
\\ \sim &   \int_{B_r(y)\cap \{ \rho^{m-1}> 2\delta\}}
|\tilde \rho(x)-\rho(x) | dx.
\end{split}
 \]
This is the whole integral if
$  x_{0,n} > ct$. In that case
$ r  \sim \sqrt{t}/x_{0,n} \sim \sqrt{t} \rho^{m-1}(y_0)$
and
\[  |B_{\sqrt{t}}^i (x_0)|_\sigma \sim  \rho^{(m-1)(2-m+n/2)}(y_0)
 t^{n/2} \sim r^{n} \rho^{2-m}(y_0)     \]
completes the estimate.
If $x_{0,n} \le 2t$ we may as well enlarge the ball and assume that
$x_{0,n}=0$. Then
\[
  \int_{H \cap B_{\sqrt{t}}^i(x_0) \cap \{ x_n < 2\delta\}}|w_2(x)-w_1(x)|dx
\le c t^{n-1} \delta^{2+\sigma} = c t^{n-1} \delta \delta^{\frac{1}{m-1}} . \]
In this case $ |B^i_{\sqrt{t}}(x_0)|_{\sigma}  \sim t^{n+\sigma}$
which again completes the estimate.
\end{proof}

\section{Proof of Proposition \ref{regularity}}

\subsection{The interior estimate}
We notice that the traveling wave is given by
$(x_n+t)_+^{\frac1{m-1}}$ so that the free boundary is the hyperplane
with equation $x_n=-t$. By the estimates on $\rho$ we know that the
positivity set of $\rho$ inside $Q$ is contained in the set $x+t\ge
-\delta$ and contains the set $x+t\ge \delta$. We take a point
$Y_0=(x_0,t_0)$ in this last set, and denote by $L=x_{0,n} + t_0$
which is close to the distance from $Y_0$ to the free boundary  $\Gamma$.  We
assume that $\delta\ll L <1$.

Take now a cylinder ${\mathcal C}_0$ centered on $Y_0$ with measures
$2L/100$ in the $t$ direction and $2mL/(m-1)100$ in the
$x$-direction. Since $p_{tw}(T_0)=cL>0 $, by our estimates we conclude
that the pressure $p$ is bounded uniformly in ${\mathcal C}_0$ from
below by $.9L$ and from above by $1.1L$.  We now introduce the
rescaling
\[
 \widehat p(t',x')=\frac1{L}p(Lt',x_0+Lx')
\]
and then $\widehat p$ is uniformly bounded by constants $0.9$ and
$1.1$ in ${\mathcal C}_0'$ which is a certain cylinder centered at
$(0,0)$ with size independent of $\delta, L$. By the theory of
uniformly parabolic equations in divergence form applied to the
corresponding density  $\widehat \rho$ (see \cite{LUS75}) we conclude that
\[
|\partial_t^k\partial_x^\alpha\widehat \rho|\lesssim  \max\{1,{t'}^{-k-|\alpha|/2}\}
\]
 holds in the interior of the cylinder, for every $k$ and  $\alpha$.

Consider now the difference $v=p-p_{tw}$. It satisfies the equation
\[
v_t=(m-1)p\Delta v + \vec{b}\cdot \nabla v, \qquad \vec{b}=\nabla p +(m/(m-1))e_n
\]
or in divergence form
\begin{equation}\label{divform}
v_t=(m-1)\nabla\cdot(p\nabla v) + \vec{b_1}\cdot \nabla v,  \qquad \vec{b_1}=\vec{b}+(1-m)\nabla p.
\end{equation}
Applying again the theory from \cite{LUS75}  and using the estimates obtained in step (i) on the rescaled version of $v$ we get uniform estimates for $\widehat v$ in $1/2{\mathcal C}_0'$ of the form
\begin{equation}
\|\partial_t^k \partial^\alpha  \widehat v\|_{L^\infty(\frac12{\mathcal C}_0')}\le  t^{-k-|\alpha|/2}  \| \widehat v\|_{L^\infty({\mathcal C}_0')}= t^{-k-|\alpha|/2} \frac1{L}\|v\|_{L^\infty({\mathcal C}_0)}=t^{-k-|\alpha|/2}  \frac{C\delta}L .
\end{equation}
Undoing the rescaling we obtain the interior estimate of Proposition \ref{deltaest}.

\subsection{Change of coordinates}

We use the notation introduced in the Section \ref{sec.vonmises}. Since we restrict our attention to balls centered at the boundary we do not need to work with
intrinsic balls and define $B_R(y)^+= B_R(y) \cap H$. We consider the problem
written as in \eqref{TPE}. The condition
\begin{equation}  |\nabla \rho^{m-1}-e_n| \le \mu/4
\end{equation}
for $\mu \le \frac12$ translates  into
\begin{equation}
\sup_H |\grady u| \leq \mu, \label{smalldata}
\end{equation}
and
\begin{equation}\label{smallsup}
\Vert u \Vert_{sup} \le 2\Vert \rho^{m-1}-(t+x_n) \Vert_{sup} \end{equation}
and the second  part of Proposition \ref{regularity} follows from
the  following  local regularity result:

\begin{lemma} \label{lem:local}
There exists $\mu>0$ so that for any
any bounded function
\[ u: (0,1] \times B^+_2(x)\cap H  \]
which satisfies  equation \eqref{TPE}  and
\[ |D_x u | \le \mu  \]
satisfies also
\[| \partial_t^k \partial_y^\alpha u (t,y)|  \le c_{k,\alpha,m,n}
t^{-k-|\alpha|}  \Vert u \Vert_{L^\infty((0,1] \times B^+_2(0))}\]
for $0<t\le 1$ and $y_n \in B_1^+(0)$.
\end{lemma}

This immediately implies half of the desired estimate in Part ii). The second half, when $y_n>t$, follows from the interior estimate of the first part:

\[| \partial_t^k \partial_y^\alpha u (t,y)|  \le c_{k,\alpha,m,n}
t^{-k-|\alpha|/2}(t+y_n)^{-|\alpha|/2}  \Vert u \Vert_{L^\infty((0,1] \times B^+_2(0))}.\]
In particular
  \begin{equation} \label{transsecond}  \Vert D_{t,x} u \Vert_{L^\infty}
    + \Vert y_n D^2_y u \Vert_{L^\infty} \le c \mu.  \end{equation}


\subsection{Real  analysis lemmata}
We collect three different real analysis results. We will encounter equations of the type
\begin{equation} \label{divergence1d}    \frac{d}{dx} (x F(x,v)) + s G(x,v) = 0 \end{equation}
 satisfying
\[
\left| \frac{\partial F}{\partial v} -1\right| + \left| \frac{\partial G}{\partial v} -1\right| < \delta.
\]

\begin{lemma}\label{verticalderivative}  Suppose that $ \frac1{s+1}<p \le \infty$ and that $\delta$ is sufficiently small. Then there exists a unique solution $v \in L^p$ to \eqref{divergence1d} which satisfies
\[ \Vert v \Vert_{L^p} \lesssim  \Vert F(.,0) \Vert_{L^p} + \Vert G(.,0) \Vert_{L^p}. \]
If
\begin{equation}    x f(x) v' + (s+1)G(x,v)= 0 \end{equation}
where
\[ |f-1| < \delta, \qquad   |\partial_v G(x,v) -1| < \delta \]
   then
\[ \Vert x_n v' \Vert_{L^p} + \Vert v \Vert_{L^p} \lesssim  \Vert G(.,0) \Vert_{L^p} \]
\end{lemma}
\begin{proof}
The linear equation
\[   (x v)' +  s v = (xg)' + f \]
has the general solution - at least for regular $f$ and $g$ -
\[  x^{-1-s}  \int_0^x y^{s} (f(y)-s g(y))   dy + c x^{-1-s}. \]
We claim that
\begin{equation} \label{claimint}  \Big\Vert  x^{-1-s}  \int_0^x y^{s} f(y) dy \Big\Vert_{L^p}
\le \frac{1}{1+s-\frac1p} \Vert f \Vert_{L^p}    \end{equation}
which implies that $c=0$ for every solution in $L^p$. An easy fixed point argument implies the first estimate and uniqueness, and existence in $L^p$ and uniqueness in the second part.
 The equation shows that in the second part
\[  \Vert x_n v' \Vert_{L^p} \le     \Vert G(.,0) \Vert_{L^p}.  \]
It remains to verify claim \eqref{claimint}. We prove the more general inequality for $s>0$
\[ \Big\Vert x^{-1-\frac1p -s}    \int_0^x y^{s+\frac1p} dx \Big\Vert_{L^p} \le \frac1{s+1} \Vert f \Vert_{L^p}. \]
It follows by interpolation from the trivial estimates
\[   \left| x^{-1-s} \int_0^x y^{s} f(y)\right| \le  \frac1{s+1}   \Vert f \Vert_{L^\infty} \]
and
\[   y^{1+s} \int_{y}^\infty x^{-s-2} dx  = \frac1{s+1}. \]
  \end{proof}

\begin{lemma} \label{interpol}
	Let $ K \subset J \subset \R $ be an open intervals with
$ |K| \ge  |J|/2 $.
	Further let
	\[ \eta: J \to [0,1] \]
have compact support, $\eta|_K=1$ with $\eta^{-1}([t,\infty)) $ connected
for $0\le t \le 1$. Suppose $\vartheta \ge 0$ and  $ u \in C^2(J) \cap L^p(J)$.
Then there exists a constant $ c = c(p) > 0 $ such that
\[ \Vert y^\vartheta \, \eta \, u' \Vert_{L^p(J)}^2 \leq c
\Vert y^\vartheta \, u \Vert_{L^p(J)} \left( |J|^{-2} \Vert y^\vartheta \eta^2 u \Vert_{L^p(J)} +  \Vert y^\vartheta \, (\eta^2 \, u)'' \Vert_{L^p(J)} \, \right).  \]
\end{lemma}

\begin{proof}  We claim
\begin{equation} \label{FriedmanInterpol}
	\Vert u' \Vert_{L^p(\R)}^2 \lesssim_p \Vert u \Vert_{L^p(\R)} \, \Vert u'' \Vert_{L^p(\R)}
\end{equation}
for $u \in W^{2,p}$  which is contained in \cite{FriedmanPDE}. It can be proven by an integration by parts argument if $p \ge 2$:
\[ \int (u')^p dx = \int |u'|^{p-2} u' u' dx = - \int (p-1) |u'|^{p-2} u'' u dx
\le (p-1)\Vert u' \Vert_{L^p}^{p-2} \Vert u'' \Vert_{L^p} \Vert u \Vert_{L^p} \]
by dividing by $\Vert u' \Vert_{L^p}^{p-2}$. A simple extension argument shows that
\begin{equation} \label{unweighted}  \Vert u' \Vert_{L^p(0,\infty)}^2 \le 3(p-1) \Vert u \Vert_{L^p(0,\infty)} \Vert u'' \Vert_{L^p(0,\infty)} \end{equation}
for $p\ge 2$ - if $1\le p< 2$ the constant will be different.  For $\vartheta>0$, by an application of Fubini, for $a>0$ and $x_0\ge 0$
\[
\begin{split}
 \Vert x^\vartheta u' \Vert_{L^p([x_0,\infty))}^p - x_0^{p\vartheta} \Vert u' \Vert_{L^p([x_0,\infty))}^p
=   \int_{x_0}^\infty  x^{p\vartheta -1} \Vert u' \Vert^p_{L^p([x,\infty))} dx
\hspace{-6cm} & \\ \le & 3(p-1) \int_{x_0}^\infty  x^{p\vartheta-1} \Vert u'' \Vert_{L^p([x,\infty))}^{p/2}
\Vert u \Vert_{L^p([x,\infty))}^{p/2} dx
\\ \le & 3(p-1) \int_{x_0}^\infty x^{p\vartheta-1} \left(  \frac{a}2 \Vert u'' \Vert^p_{L^p([x,\infty))}
+ \frac1{2a} \Vert u \Vert^p_{L^p([x,\infty))}\right)  dx
\\ = & (3(p-1))^{p/2} \Big( \frac{a}{2} \Big( \Vert  x^{\vartheta}|u''|\Vert^p_{L^p([x_0,\infty)) }  - \Vert x_0^{\vartheta} u'' \Vert^p_{L^p([x_0,\infty))}\Big)
\\ & + \frac{1}{2a}\Big( \Vert  x^{\vartheta} |u|^p \Vert_{L^p}^p - \Vert x_0^\theta u \Vert_{L^p([x_0,\infty))}^p
\Big).
\end{split}
\]
 We optimize $a$ and combine the estimate with \eqref{unweighted} to arrive at
\[ \Vert x^\vartheta u' \Vert_{L^p(x_0,\infty)}^2 \le 3(p-1) \Vert x^\vartheta u \Vert_{L^p(x_0,\infty)} \Vert x^{\vartheta} u'' \Vert_{L^p(x_0,\infty)}. \]
By extension to the right we obtain
\begin{equation} \label{OurInterpol}
	\Vert x^\vartheta u' \Vert_{L^p(I)}^2 \lesssim_p \Vert x^\vartheta u \Vert_{L^p(I)} \left( \, |I|^{-2} \, \Vert x^\vartheta u \Vert_{L^p(I)} + \Vert x^\vartheta u'' \Vert_{L^p(I)} \, \right)
\end{equation}
for any bounded interval $ I $. Again using Fubini
\[ \begin{split}
\Vert y^\vartheta \eta u' \Vert_{L^p}^p
= & \int_0^1 \Vert y^\vartheta u' \Vert_{L^p( \eta > t^{1/p})} dt
\\ \le & c\int_0^1 \Vert y^\vartheta u \Vert_{L^p(\eta> t^{1/p} )}^p
\Big( |J|^{-2} \Vert y^\vartheta u \Vert_{L^p(\eta> t^{1/p})}
+ \Vert y^\vartheta u''  \Vert_{L^p(\eta> t^{1/p})}\Big) dt
\\ \lesssim & \Vert y^\vartheta  u \Vert_{L^p(J)} \Big(
|J|^{-2} \Vert y^\vartheta \eta^2 u \Vert_{L^p(J)} +
\Vert y^\vartheta (\eta^2 u)'' \Vert_{L^p(J)} \Big),
\end{split}
\]
we complete the proof of the lemma.
\end{proof}

The third tool  is a Calder\'on-Zygmund type estimate, \cite[Proposition 3.23]{Kienz14}:

\begin{lemma}\label{CZ}  Suppose that $p>(1+\sigma)^{-1}  $ and
\[ f \in L^p(\R \times H).  \]
Then there exists $v$ satisfying
\[ \Vert D_{s,y} v \Vert_{L^p(\R \times H)} + \Vert x_n D^2_x v \Vert_{L^p(\R \times H)} \le c \Vert f \Vert_{L^p(\R\times H)} \]
and
\[  v_t - x_n \Delta v -\sigma \partial_n v = f. \]
Moreover $v$ is unique up to the addition of a constant.
Similarly, if $ F^i \in L^p$ then there is a unique $v$ with $ \nabla v \in L^p$
and
\[ v_t - x_n \Delta v - \sigma \partial_n v = x_n^{-s} \partial_i  x_n^{1+s} F^i. \]
\end{lemma}


\subsection{Improved estimates}

After this preparation we prove Lemma \ref{lem:local}
for solutions satisfying \eqref{TPE} and \eqref{TPE_coeff}.
Since we can choose $ \mu $ smaller if need be,
we can assume that $ \mu_1 $ and hence $ |\grady u| $ are as
small as we like.
Then also
\[ |a^{ij} - \delta^{ij}| + |b^j - \delta^{nj}| \leq C  |\grady u| \leq C\mu_1, \] with a constant depending only on the space dimension, $m$
and
\[ |\gradsy a^{ij}| + |\gradsy b^{j}| \le  C  (1 + |\grady u |) \, |\grady \gradsy u| \leq C (1 + \mu_1) \, |\grady \gradsy u| \]
on $ (s_1,s_2) \times \overline{H} $.

We define a cutoff function
\[ \eta(s,y) =  s \ \prod_{j=1}^{n} \eta_0(y_j) \]
where
\[ \eta_0 \in C^\infty_0(-2,2),\quad \eta_0 \le 1, \quad  \eta_0=1 \text{ on } (-1,1) \quad
 (\eta_0)^{-1}(\{ (\sigma,\infty) \})  \text{ is connected for } \sigma \le 1. \]

  A direct calculation shows that
\[ w \ldef \eta^2 \, u  \] satisfies the equation
\begin{equation} \label{truncated}
		\ds w - y_n \, \lapy w - (1 + \sigma) \, \dyn w = F[u,w,\eta]
\end{equation}
on $ I_r \times \overline{H} $ with zero initial value, where
\begin{equation*} \label{truncatedInhom}
\begin{split}
	F[u,w,\eta] \ldef y_n & \, (a^{ij} - \delta^{ij}) \, \dij w + (1 + \sigma) \, (b^j - \delta^{nj}) \, \dj w \\
		& + u \left( \ds (\eta^2) - y_n \, a^{ij} \, \dij (\eta^2) - (1 + \sigma) \, b^j \, \dj (\eta^2) \right) \\
		& - y_n \, (a^{ij} + a^{ji}) \, \di (\eta^2) \, \dj u.
\end{split}
\end{equation*}
We fix $p =2 \max\{ (1+\sigma)^{-1},n+1\}$,
 apply Lemma \ref{CZ} to the function $w$ in  \eqref{truncated}
and obtain
\[
\begin{split}
 \Vert \gradsy w \Vert_{L^p(\R \times H)} + \Vert y_n \, D^2_y w \Vert_{L^p(\R \times H)} \lesssim_{n,\sigma}&   \Vert F[u,w,\eta] \Vert_{L^p(\R \times H)}
\\ \lesssim_{n,\sigma} & \mu \left( \Vert \gradsy w \Vert_{L^p(\R \times H)} + \Vert y_n \, D^2_y w \Vert_{L^p(\R \times H)}\right)
\\ & +  \Vert |\partial_s \eta^2|  + |y_n D^2_y \eta^2|+ |D_y \eta^2| \Vert_{L^\infty(\R \times H)} \Vert u \Vert_{L^\infty}
\\ & + \Vert y_n D_y (\eta^2) D_y u \Vert_{L^p}.
\end{split}
\]
By Lemma \ref{interpol}
\begin{equation}\label{interpolation1}    \Vert y_n D_y (\eta^2) D_y u \Vert^2_{L^p}
\lesssim \Big(\Vert u \Vert_{L^\infty} + \Vert y_n D^2_x (\eta^2 u) \Vert_{L^p}\Big)  \Vert u \Vert_{L^\infty}
\end{equation}
hence
\begin{equation}   \Vert y_n D_y (\eta^2) D_y u \Vert_{L^p}
\lesssim \varepsilon \Vert y_n D^2_x w  \Vert_{L^p}
+ C(\varepsilon) \Vert u \Vert_{L^\infty}.
\end{equation}

The first term and the third term on the right hand side can be controlled by the left
 hand side if we choose $\varepsilon$ and $\mu$ sufficiently small, $\mu,\varepsilon \le \mu_0$ with $\mu_0$ depending only on $n$ and $ \sigma$, which is a function of
 $m$. The first factor of the second term is bounded and gives the
 desired supremum norm of $u$.
 Altogether
\begin{equation}  \Vert \nabla_{s,y} u \Vert_{L^p([1,2]\times B_1(0)\cap H )}
+ \Vert y_n D^2_y u \Vert_{L^p([1,2]\times B_1(0)\cap H )} \le c
\Vert u \Vert_{L^\infty((0,2)\times B_2(0)\cap H}.
\end{equation}

\subsection{Higher order derivatives}
We deal inductively with higher order derivatives. We assume that $k\ge 0$
and
\begin{equation}   \Vert D_{t,y}^{k+1} u \Vert_{L^p((0,2] \times B_2(0))}
+  \Vert y_n D^2_x  D_{t,y}^k \Vert_{L^p((0,2) \times B_2(0))}
\le \delta \end{equation}
and we claim that then
\[   \Vert D_{t,y}^{k+2} u \Vert_{L^p((0,2] \times B_2(0))}
+  \Vert y_n D^2_x  D_{t,y}^{k+1} \Vert_{L^p((0,2) \times B_2(0))}
\lesssim  \delta. \]
Let $\alpha$ be a multi-index of length $k$ and $v = \eta^2\partial^\alpha u$. It satisfies
\begin{equation} \label{truncated2}
		\ds w - y_n \, \lapy w - (1 + \sigma) \, \dyn w = F_\alpha[u,w,\eta]
\end{equation}
where
\begin{equation*} \label{truncatedInhom2}
\begin{split}
	F_\alpha[u,w,\eta] \ldef y_n & \, (a^{ij} - \delta^{ij}) \, \dij w + (1 + \sigma) \, (b^j - \delta^{nj}) \, \dj w \\
		& + \partial^\alpha u \left( \ds (\eta^2) - y_n \, a^{ij} \, \dij (\eta^2) - (1 + \sigma) \, b^j \, \dj (\eta^2) \right) \\
		& - y_n \, (a^{ij} + a^{ji}) \, \di (\eta^2) \, \dj \partial^\alpha u \\
& + G_\alpha
\end{split}
\end{equation*}
where $G_\alpha$ contains all the terms with at least two factors with
at least two derivatives.

We begin with the considerations for $|\alpha|=1$, $1\le i \le n-1$. Then
\[
\begin{split}
\Vert G_\alpha \Vert_{L^p} \lesssim  &  \Vert y_n \eta^2 D^2 u D\partial_i u \Vert_{L^p} + \Vert y_n \eta D\eta Du D \partial_i u \Vert_{L^p}
\\  \lesssim  & c \Vert    D^2 \eta^2 u \Vert_{L^p} \Vert x_n D^2_x u \Vert_{L^\infty}\\ \lesssim & c \mu \Vert x_n \eta^2 D^3_y u \Vert_{L^p}.
\end{split}
\]
The very same argument works for $w= \eta^2 \partial_t u$. It is easy to make the argument rigorous by using finite differences.

Let $D'=D'_{t,x}$ denote all derivatives besides the one in direction $e_n$. Then we have proven
\[  \Vert D_{t,x}  D' u \Vert_{L^p([0,1] \times B_1(0))}
+ \Vert x_n D^2_{x}  D' u \Vert_{L^p([0,1] \times B_1(0))}
\le c \mu.
\]

To control the vertical derivative we recall that
\begin{equation}  \partial_n \Big[x_{n}\frac{\partial_n u + |\nabla' u|^2}{1+ \partial_n u } \Big]+  s \Big[x_{n}^{s+1}\frac{\partial_n u + |\nabla' u|^2}{1+ \partial_n u } \Big] =  u_t - x_n \Delta' u
\label{vertical}
\end{equation}
hence, with $v= \partial^2_n u$
\[ \partial_n \Big[ x_na^{nn}  v\Big]  + (s+1) a^{nn} v  )
= \partial^2_{nt} u -  \Delta' u - x_n \Delta' u
-\partial_n \Big[ x_na^{n\alpha} \partial^2_{\alpha n} u  \Big]
- s a^{n\alpha} \partial^2_{n\alpha} u.
 \]
By Lemma \ref{verticalderivative}
\[ \Vert Dv \Vert_{L^p} \lesssim  \Vert DD'u \Vert_{L^p} + \Vert x_n D^2 u \Vert_{L^p}.
\]
Now we rewrite the derivative term on the right hand side
\[\partial_n \Big[ x_na^{n\alpha} \partial^2_{\alpha n} u  \Big]
= a^{n\alpha} \partial^2_{\alpha n} u + x_n D^2 u D D' u + x_n a^{n\alpha} \partial^2_{\alpha n} v, \]
and use the second estimate to conclude
\[ \Vert x_n D^2 v \Vert_{L^p} \le c \Vert DD'u \Vert_{L^p} + \Vert x_n D^2 u \Vert_{L^p}.
\]
We may choose $p$ as large and hence
\[ \Vert u \Vert_{C^{1,\alpha}} + \Vert x_n D^2 u \Vert_{sup} \le \mu. \]

We iterate the argument. For second tangential derivatives we have to bound terms like
\[  \Vert x_n D^2 u D^2 u  \Vert_{L^p} \le \Vert x_n D^2 u \Vert_{L^\infty}
\Vert D^2 u \Vert_{L^p} \le c \mu \Vert x_n D^3 u \Vert_{L^p} \]
with at least one tangential derivative (or difference quotient for the rigorous proof) on the RHS. Next we differentiate \eqref{vertical} twice in the vertical
direction. The details are similar but simpler than  before.

\section{Estimates for global compactly supported solutions}\label{sec.thm2}

The proof of Theorem 2 follows from the preceding analysis but it is
not immediate because we want to establish $C^\infty$ regularity after
a time $T(\rho_0)$ that can be estimated in terms of simple
information on the initial data. This is a rather precise result that
needs a careful quantitative analysis. Such an analysis is interesting in
itself. It occupies this section as follows.  First, we review the
Barenblatt solutions, and the  qualitative result on convergence
of any nonnegative, integrable solution $\rho(t,x)$ towards the
Barenblatt profile with the same mass. Then we start the quantitative
analysis of sizes of solutions and location of the free
boundaries. These solutions do resemble the Barenblatt solution but
for constant factors that must be controlled. The theory of entropy
and entropy dissipation allows us to transform this resemblance into
convergence with rate as time goes to infinity in a very precise
way. We use these results
and the flatness criterion of Theorem \ref{th-main} to obtain
$C^\infty$ regularity for large times.  We recall the definition of $\lambda$ in \eqref{lambda}.

\subsection{The Barenblatt solution and plain  asymptotic convergence}

The Barenblatt solution is given by the formula

\begin{equation} \label{Barenblatt}
{\mathcal B}(t,x;M) =   t^{-n\lambda}   F(\sqrt{\lambda} xt^{-\lambda}), \qquad
F (\xi)= \Big( A^2 - \frac{|\xi|^2}2 \Big)^{\frac1{m-1}}_+
\end{equation}
where $\lambda$ is the above-mentioned constant and $A>0$ is a free constant that can be easily calculated in terms of the mass $M$ of the solution,
\begin{equation}\label{relmass}
A=aM^{(m-1)\lambda}\,,
\end{equation}
and $a=a(m,n)$ is given by
\begin{equation}  a =  \Big(\int_{B_{\sqrt{2}}(0)} (1-\frac12 |\xi|^2)^{\frac1{m-1}} d\xi \Big)^{\frac1{n\lambda}}  \label{defa}     \end{equation}
see the details in \cite{Vsmooth}, Section 2.1. Recall that writing the equation as $\rho_t=k\Delta \rho^m$ with a constant $k=(m-1)/m$ changes the coefficient in the expression of $F$ written in that reference (in fact, it simplifies it).

The Barenblatt solutions play an important role in describing general global nonnegative solutions for large times. This is reflected in the result on asymptotic convergence of any global solution $\rho(t,x) $ of the PME with nonnegative, integrable initial data towards the Barenblatt solution with same mass $M$, cf. \cite{vazquezPME}, Theorem 18.1. The uniform convergence result, formula (18.7) of that reference, says that
\[
\lim_{t\to\infty} t^{n\lambda} \sup_x |\rho(t,x)-\mathcal{B}(t,x;M)|=0\,,
\]
with $\lambda$ as defined above and $M=\int_{\Rn}\rho_0(x)\,dx>0$. It is known that this initial mass is conserved in time.

It is convenient to reformulate the result as uniform convergence in space-time around the time $t=1$ for some rescaled solutions. Indeed, if we define for $k>1$ the family of rescalings
\begin{equation}\label{k.scale}
\rho_k(t,x)=k^{n\lambda}\rho(kt, k^{\lambda}x)\,,
\end{equation}
then the $\rho_k$ are again solutions of the PME with the same mass $M$. The previous convergence result can be equivalently stated as the uniform convergence
\begin{equation}\label{resc.conv}
\lim_{k\to\infty}  |\rho_k(t,x)-\mathcal{B}(t,x,M)|=0\,
\end{equation}
in every cylinder of the form $ \widehat Q=(t_1,t_2)\times \Rn$ with $0<t_1<t_2$.
It is also proved that for compactly supported data the free boundary $\Gamma_k$ of $\rho_k$ converges uniformly to the free boundary $\Gamma(\mathcal{B})$ of $\mathcal{B}$, which is  an expanding ball with  radius
\begin{equation}
{\mathcal R}(t)=R_B (M^{(m-1)}t)^{\lambda}
\end{equation}
with constant $R_B=a(2/\lambda)^{1/2}$, a function of $m$ and $n$.

\subsection{Quantitative question. Reduction and first bounds}

We need to prove a quantitative version of the error that is committed in such approximation for a class of initial data and for $t$ large enough. This is better done after some reduction of the problem based on the scale invariance of the equation. Let $M=\int_{\Rn}\rho_0(x)\,dx>0$ be the mass of the initial data, that is
conserved in time, and let $\rho_0$ be supported in the ball of radius $R$.
Then, by defining
\[
\widetilde \rho(t,x)  =   M^{-1}  \rho(M^{1-m} R^{2} t, R x)
\]
we get yet a solution of the same PME, but now it has mass $1$ and $\widetilde \rho(0,x)$ is supported in $B_1(0)$. Thanks to this transformation we may assume
without loss of generality that $\rho_0$ has mass $1$, and is supported in a ball of radius $1$. Also, by space translation we may also assume that $\rho_0$ has center of mass $x=0$. Let us call this class of solutions $\mathcal C$.
%
%
%

 Let us start by obtaining uniform estimates for the whole class  $\mathcal C$, and let us see how these estimates look like the values for $\mathcal B(t,x;1)$, at least when $t$ is large enough.

\subsubsection{Sup estimates for the solutions} The so-called $L^1$--$L^\infty$ smoothing effect with best constant says that
\begin{equation}\label{sup.up}
\sup_x \rho(t,x)  \le  \mathcal{B}(t,0;1) =    A^{\frac2{m-1}}    t^{-n\lambda }\,,
\end{equation}
 cf.  \cite{Vsym}. This holds for every $t>0$.

\subsubsection{Bounds from below} The next step consists in deriving lower bounds on the solution $\rho(t,x)$. We use the results by Aronson and Caffarelli in \cite{AC83}, according to which there exists a constant $C>0$ depending only on $n$ and $m$ such that every nonnegative global solution satisfies the inequality (rather, family of inequalities)
\begin{equation}\label{ac83}
\int_{B_r(x_0)} \rho_0(x)\,dx\le C\, \left(r^{\frac1{\lambda(m-1)}}t^{-\frac{}{m-1}}+ t^{\frac{n}2}\rho^{\frac1{2\lambda}}(x_0,t) \right)
\end{equation}
for every $x_0\in \Rn$, $r>0$, and $t>0$. Therefore,  if $\rho_0$ belongs to the class $\mathcal C$ and if $r>|x_0|+1$,  the left-hand side of the formula is just $1$, so that for
$$
t \ge (C/2)^{m-1}r^{1/\lambda}\,,
$$
it follows from \eqref{ac83}  that $1/2 \le Ct^{n/2}\rho^{1/2\lambda}(t,x_0)$, or
\begin{equation}\label{sup.bl}
\rho(t,x_0)\ge C_1 t^{-n\lambda}
\end{equation}
a size to be compared with \eqref{sup.up} and with the Barenblatt solution. Since the time condition can be written as
\begin{equation}
|x_0|+1 \le C_2 t^{\lambda}\,,
\end{equation}
these estimates show that $\rho(t,\cdot)$ is larger than  a Barenblatt solution of small but comparable mass at the same time, and this holds for all  $t\ge t_1$ with fixed $t_1>0$. Such  estimate  is not precise enough in the constants, but it gives the correct dependence in time, and it is uniform for all the class of data we consider.

\subsubsection{Support estimates from below} This time they will be sharp. We define $R(t)$ as the smallest real  number such that $\rho(t,.)$ is supported in $B_{R(t)}$. By known theory this radius is monotonically increasing in time. Again, by a result of the last author based on symmetrization the support is at least as large in measure as the one of the Barenblatt solution and hence
\begin{equation}\label{bounds}
R(t) \ge R_B \, t^{\lambda},
 \end{equation}
 where $R_B= R_B(m,n)$ is the radius of the unit Barenblatt solution at $t=1$.

\subsubsection{Upper bounds on $R(t)$} They are more difficult and not so accurate. B\'enilan, Crandall and Pierre have introduced in \cite{BCP84} the weighted norm
\begin{equation}\label{def.mu}
 \Vert \mu \Vert_{r,m} = \sup_{R\ge r}  R^{-\frac2{m-1}-n}   \mu(B_R(0))
\end{equation}
for given $r\ge 1$, and the corresponding end time $T(\rho_0)=c_1(n,m) /\Vert \rho_0 \Vert_{r,m}^{m-1}$. Then their estimate (1.7) asserts that for all times $0 < t < T(\rho_0)$ we have the very explicit upper estimate:
\begin{equation*}
t^{n\lambda} \rho(x,t) \le
c_2(n,m)\,  r^{\frac{2}{m-1}} \Vert \rho_0 \Vert_{r,m}^{2\lambda}
\end{equation*}
if  $|x| \le r$. As explained in \cite{CVW87}, Lemma 1.2, the restriction $r\ge 1$ is unimportant since it can be eliminated by rescaling, but in letting $r\to 0$ we have to be careful
with the possible divergence of the quotient in the right-hand side of formula \eqref{def.mu}.
The way to make this term finite is to shift the origin of coordinates to a point $x_0$ away from the support of $\rho_0$, i.e., $|x_0|\ge 1$.
Then we use the shifted norm $ \Vert \mu \Vert_{x_0r,m}$ and get the estimate
\begin{equation}\label{mu.est}
t^{n\lambda} \rho(x,t) \le
c_2(n,m)\,  r^{\frac{2}{m-1}} \Vert \rho_0 \Vert_{x_0,r,m}^{2\lambda}.
\end{equation}
if $0<t<T(\rho_0)=c_1 /\Vert \rho_0 \Vert_{x_0,r,m}^{m-1}$ and $|x-x_0|\le r$, $r>0$.
In particular, letting $r\to 0$ for fixed $t$ we get the result: if $|x_0|>1$ then $\rho(t,x_0)=0$ if
\[  t \le c_3(n,m) (|x_0|-1)^{n(m-1)+2}\,.
\]
In this way we get the upper bound
\begin{equation} \label{bounds.1}
 R(t) \le 1+ (t/c_3)^{\lambda}.
\end{equation}
It is an important feature of this bound that the right-hand side amounts for large $t$  (at least $t>2$) to at most a fixed constant $L$ times the lower bound, which uses the Barenblatt radius.

\subsubsection{Roundness of level sets via moving planes}
Using comparison and moving plane techniques, Caffarelli et
al. \cite{CVW87} have shown that $B_{R(t)-2}$ is contained in the
support and the following monotonicity holds; if
\[ 1 \le \frac{|y|^2-|x|^2}{2|y-x|} \qquad \text{ then } \rho(y,t) \le
\rho(x,t) .\]
In particular, if $r(t)>1$ and $|x_0|>r$ then the level
sets
\[
S= \{x: \rho(t,x) = \rho(t,x_0)\}
\]
are graphs over the sphere, i.\,e., if $\theta(x)$ are spherical angle
coordinates in $\Rn$, then there exists a unique function $h$ such
that
\[  S = \{ x : |x|= h(\theta(x))\}
\]
and we also have : $ \max_\theta h -\min_\theta h \le 2 $. Write $h(x)=h(\theta(x))$.
 Then  for $x,y \in S$
\[   |h(y)|^2-|h(x)|^2 \le 2|x-y| \]
\[  2h(x)(h(y)-h(x)) \le 2 |h(y)-h(x)| + 2 h(x) |x/|x|-y/|y||+ O(|x-y|^2)  \]
\[ |x||\nabla h(x)| \le  \frac{h(x)}{h(x)-2} .
\]
In particular, the level sets become rounder with time, both in the
supremum norm, and in the Lipschitz norm, to finally look like balls
as $t\to\infty$.

\subsubsection{The gradient bound for the pressure.}
The upper and lower bounds for the pressure, plus the monotonicity
imply a uniform gradient bound for the pressure. We follow the proof
by Caffarelli et al. \cite{CVW87}, Theorem 1. Let $t_1$ be a time at
which the solutions of $\mathcal C$ have expanded to cover twice the
unit ball, a time that can be uniformly estimated according to the
preceding paragraphs. Then we have

\begin{lemma} If $\rho\in {\mathcal  C}$ and $t\ge kt_1$,  there is a uniform constant $C$ such that
\begin{equation}
|\nabla \rho^{m-1} (t,x) | \le C \,t^{-(1-\lambda) }\,.
 \end{equation}
The constant $k$ is also uniform, i.e. it depends only on $m$ and $n$.
 \end{lemma}

\begin{proof}  (i) To see this we take pressure  $v:= \rho^{m-1}$ and define
\[
v_\varepsilon(t,x) = (1+\epsilon)^{-1} v((1+\varepsilon)t+t_0, (1+\varepsilon)x)
\]
with $t_0>t_1$ conveniently chosen.  We show that this new variable satisfies
\[
v_\varepsilon(0,x) \le v_{\varepsilon=0}(0,x)=v(t_0,x)
\]
if we put $t_0=kt_1$ with a constant $k\ge 2$ that is estimated below. This comparison can be done as a variant of the proof of \cite{CVW87}, Theorem 1. We add a sketch of the comparison argument with the novelties that allow to obtain uniform constants. We write
$$
v_\varepsilon(0,x)-v(t_0,x)=-\frac{\varepsilon}{1+\varepsilon}v(t_0,(1+\varepsilon)x) +
v(t_0,(1+\varepsilon)x )- v(t_0,x).
$$
The first term in the r.h.s. is negative so we only have to consider the last one.
Now for $|x|>1$ (i.e., outside of the initial support), the monotonicity condition proved in  \cite{CVW87} implies that $v(t_0,(1+\varepsilon)x )\le  v(t_0,x)$.

We have to examine carefully what happens for $|x|\le 1$. By the estimates of previous paragraphs we  know that in the cylinder $Q_1=(kt_1/2, kt_1)\times B_{R_1}(0)$,  $R_1=C(kt_1)^{\lambda}$ we have
$$
0 <c_1(k)< v(t,x)\le c_2(k)
$$
and the ratio between the maximum and minimum is bounded uniformly. Interior estimates for uniformly parabolic equations with smooth coefficients imply then that $|\nabla v|\le K_1 c_2/R_1$ with uniform $K_1$. In this way we get for $|x|\le 1$
$$
v(t_0,(1+\varepsilon)x )- v(t_0,x)\le |\varepsilon \,x| \,|\nabla v|\le  K_1 \varepsilon \frac{ \|v\|_{L^{\infty}} } {R_1}.
$$
This means that for $k$  not so small $v(t_0,(1+\varepsilon)x )- v(t_0,x)\le \epsilon c \|v(t_0,\cdot)\|_\infty$ with a very small $c$. We conclude that
$ v_\varepsilon(0,x)-v(t_0,x)\le 0$ in $\Rn$.

\medskip

(ii) Once $v_\varepsilon(0,x) \le v_{\varepsilon=0}(0,x)$, by the Maximum Principle this is true for all $t>0$, $v_\varepsilon(t,x) \le v_{\varepsilon=0}(t,x)$. Differentiation with respect to $\varepsilon$ at $\varepsilon=0$ gives
\[
v(t+t_0,x)  - t v_t(t+t_0,x) - x \cdot \nabla v (t+t_0,x) \ge 0
\]
at least in the distribution sense, and in fact almost everywhere by known regularity theory. Since we also have \cite{vazquezPME}
\[ v_t = |\nabla v|^2 + (m-1) v \Delta v \ge |\nabla v|^2 -\frac{n}{n+ 2(m-1)}\frac{v}{t}\,,
 \]
we obtain the a.e. inequality
\[
|\nabla v|^2 \le  \left( \frac{n\lambda}{t}+ \frac1{t-t_0} \right)  v + \frac{|x|}{t-t_0} |\nabla v|.
\]
Compare with formula (2.5) of \cite{CVW87} (page 381) which is not so precise since uniformity was not under study. The gradient bound of the Lemma is an immediate consequence of this formula and the uniform bounds $v=O(t^{-n(m-1)\lambda})$ and $R(t)=O(t^{\lambda})$.
\end{proof}

\subsection{Continuous rescaling and entropy}

Pursuing again the idea of making asymptotic calculations in renormalized settings where the solution does not tend to zero but evolves into some nontrivial profile, and copying from the sizes of the Barenblatt solution we introduce the continuous scaling \cite{vazquezPME}

\begin{equation}\label{cont.scal}
 \widetilde \rho (s,y) = e^{sn}  \rho(e^{s/\lambda}  , \lambda^{-1/2} e^{s} y)
\end{equation}
with inverse
\begin{equation}
 \rho(t,x)=t^{-n\lambda} \widetilde \rho (\lambda \log(t),\lambda^{1/2} xt^{\lambda})\,,
\quad s=\lambda \log(t) \,.
\end{equation}
It is easy to see that $\widetilde \rho$ satisfies the equation
\begin{equation}
\partial_s  \widetilde \rho =\nabla_y\cdot \Big(\widetilde \rho \nabla_y(\widetilde \rho^{m-1})+  y \widetilde \rho\Big)=
\nabla_y\cdot \left(\widetilde \rho \,\nabla_y\Big[\widetilde \rho^{m-1}+\frac12 |y|^2\Big]\right)\,.
\end{equation}\normalcolor
Notice that the Barenblatt solutions transform into the following family of stationary solutions of the right-side
\begin{equation}
\widetilde F^{m-1}(y)=(A^2-\frac12 |y|^2)_+
\end{equation}
As a consequence of our previous analysis we know that $\widetilde \rho$ satisfies a number of estimates:

\begin{enumerate}
\item  It is bounded from above $\widetilde \rho\le A^{\frac2{m-1}}$;
\item  For $s\ge s_1$ the support of $\widetilde \rho(s,  \cdot)$ contains  a ball of radius $C_2$ and is contained in the ball of radius $L C_2$, both centered at $0$;
\item The solution is bounded below in the form
$$
\widetilde \rho(s,y)\ge C_3 \quad \mbox{\rm for}\quad |y|\le C_2/2, \ s\ge s_1;
$$
\item Gradient bound:
\[
|\nabla_y \widetilde \rho^{m-1}  | \le C_4, \qquad \ s\ge s_1.
\]
 All the constants are uniform for the class $\mathcal C$, they depend on $n$ and $m$.
\end{enumerate}

\subsubsection{Properties of the  entropy}  The entropy is defined by
\begin{equation}\label{def.entr}
H(\widetilde \rho ) = \int_{\R^n} \Big(  \frac{\widetilde \rho^m}{m} + \frac{1}{2}  |y|^2
\widetilde \rho\Big) \,dy
 \end{equation}
\normalcolor
and the entropy dissipation by
\begin{equation}\label{def.entrdiss}
I(\widetilde \rho) =\int  | y + \nabla (\widetilde \rho^{m-1})|^2 \widetilde \rho\,dy.
\end{equation}
They  are functions of the new time $s=\lambda \log(t)$.  They satisfy
\[
\frac{d}{ds} H(\widetilde \rho) = -I(\widetilde \rho), \qquad
 \frac{d}{ds} I(\widetilde \rho) \le  -2 I(\widetilde \rho)\,,
 \]
 see Carrillo and Toscani \cite{CaTo00} where the coefficients are a
 bit different without affecting the result. Both $H$ and $I$ are
 decreasing functions, and we also know that $H(0)\le c(n,m)$. It is
 also shown that $H(\widetilde \rho(s))$ is bounded below by the
 entropy of the stationary state $H(\widetilde F)$ that has the same
 mass, here set to 1.  More detailed asymptotic information is
 obtained by integration
$$
H(\widetilde \rho(s))+\int_{s_0}^s I(\widetilde \rho(\tau)))\,d\tau=H(\widetilde \rho(s_0)),
\quad I(\widetilde \rho(s))\le I(\widetilde \rho(s_0))e^{-2(s-s_0)}\,.
$$
It follows that
\[
I(\widetilde \rho(s))  \le     c(n,m)e^{-2s}, \qquad
H(\widetilde \rho(s))- H(F) \le  c(n,m)e^{-2s}\,,
\]
since we already know that $\widetilde \rho(s)$ converges to the Barenblatt profile $F=\widetilde{ \mathcal B}$ of mass 1.

\subsubsection{Application to our problem}   We take $0<\ve  \le \ve_{n,m}$ and work for any large time in the convex round set
\[
U_\ve (s) = \{ x : \widetilde \rho^{m-1}(s,y) > \ve  \}\,,
\]
which is uniformly bounded. By the definition of the entropy dissipation and our decay inequality for it  we have
\[
\int |\nabla \widetilde \rho^{m-1}+ y|^2 dx \le c\ve ^{-1} e^{-2s}.
\]
By the Poincar\'e's inequality there exists a constant $c_0(s,\ve )$ such that
\begin{equation}\label{poincare}
\Vert \widetilde \rho^{m-1}+\frac{1}{2}|y|^2-c_0(s,\ve)) \Vert_{H^1(U_\ve )}
\le c\ve ^{-1/2} e^{-s}.
 \end{equation}
We interpolate with the Lipschitz bound $|\nabla \widetilde \rho^{m-1}|\le C_3$ to get
\[
\Vert \widetilde\rho^{m-1} -(c_0^2-\frac{1}{2}|y|^2) \Vert_{L^\infty(U_\ve )} \le c \ve ^{-\frac1n} e^{-\frac{2}ns},
\]
if $ n >2$. Thus,  on the set $\{u>\ve \}$ we have a very precise estimate
\begin{equation}\label{ineq.exp}
 c_0^2 - c\ve ^{-\frac1n}e^{-\frac{2}{n}s} -   \frac{1}{2}  |y|^2   \le \widetilde \rho^{m-1}(s,y) \le c_0 + c\ve ^{-\frac1n} e^{-\frac{2}ns} -\frac{1}{2}|y|^2.
\end{equation}
This means that for $t$ large $\widetilde\rho$ looks like a Barenblatt profile on the set $U_\varepsilon= \{\widetilde\rho>\ve \}$, that in turn must look like a ball $B_{R_0(t)}$. The calculations are a bit different in lower dimensions, we get
\[
\Vert \widetilde\rho^{m-1} -(c_0^2-\frac{1}{2}|y|^2) \Vert_{L^\infty(U_\ve )} \le c \ve ^{-\frac12} e^{-s} \]
 if $n=1$, while for $n=2$ the expression is
\[ \Vert \widetilde\rho^{m-1} -(c_0-\frac{1}{2}|y|^2) \Vert_{L^\infty(U_\ve )} \le c \ve ^{-\frac12} e^{-s}(s+|\ln \ve| )\,.
\]

\subsubsection{Improved upper bounds on $c_0$.} We already know that at the maximum value $\widetilde \rho$ must be  bounded above by the Barenblatt profile $\widetilde F$, hence from \eqref{ineq.exp}  we have the upper bound (for $n\ge 3$)
$$
c_0^2- c\ve ^{-\frac1n} e^{-\frac{2}ns}\le A^{2}.
$$
This allows to improve the upper bound for the solution to the form
\begin{equation}\label{uppwe.vebound}
\widetilde \rho^{m-1}(s,y) \le \max \Big\{A^2 + 2c\ve ^{-\frac1n} e^{-\frac{2}ns} -\frac{1}{2}|y|^2\,, \ve^{m-1}\Big\}\,,
\end{equation}
The upper bound immediately implies that for $s$ large $U_\varepsilon \subset B_r(0)$ with
\begin{equation*}
 r = \sqrt{2}c_0 + c \varepsilon^{-\frac1n} e^{-\frac2n s}
\end{equation*}
if $n > 2$, and
\begin{equation*} 
r= \sqrt{2}c_0 + c \varepsilon^{-\frac12} e^{-s}
\end{equation*}
if $n=1$ with a logarithmic correction if $n=2$. Recall now the upper approximation for $c_0$
and we get an approximation of  $U_\varepsilon$ to the Barenblatt radius.

We point out a difficulty in taking the limit in the upper bounds
of the sets $U_\ve$ as $\ve\to 0$, since the support may have a thin
tail where $\rho$ is smaller than the error that we have
calculated. We already have a uniform upper bound for the support in a
possibly larger ball $B_{R_1}(0)$. We still have to prove below that
such an external region $B_{R_1}\setminus U_\ve (s)$ (the `tail') is
small, but in any case we know that $\widetilde\rho\le \ve $ there.

\subsubsection{An upper bound of the radius} We turn to an upper estimate of the support, for which we have to bound the possible 'tail' where $\tilde \rho $ is small. We do this by a
comparison argument with a Barenblatt solution outside the Barenblatt
radius for the Barenblatt solution with same center and mass.
We recall  that (see \eqref{bounds.1} )
\begin{equation} \label{firstupperbound}
 \supp \rho(t,.) \subset B_{c t^{\lambda}} (0) \end{equation}
for $t\ge t_0$, $t_0$  and $c$ depending only on $n$ and $m$. Then
\begin{equation}    \rho(t,x) \le \varepsilon t^{-n\lambda}  \end{equation}
provided
\begin{equation}   |x| \ge r(t)t^{\lambda}   :=  R_B t^\lambda (1+ c\varepsilon^{-\frac12-(\frac{m-2}{m-1})_+}t^{-\lambda}    + c\varepsilon^{m-1} ) \end{equation}
if $n \le 2$ (again with logarithmic correction if $n=2$ and $m \le 2$)
resp.
\begin{equation}     |x|\ge r(t)t^\lambda := R_B t^\lambda (1+ C(
\varepsilon^{-\frac12- (\frac{m-2}{m-1})_+}t^{-\lambda}  + \varepsilon^{-\frac1n} t^{-\lambda/n}
  + \varepsilon^{m-1})      \end{equation}
for $ n>2$. This information suffices to construct a comparison solution
by rescaling and time translating a Barenblatt solution.
We fix  $\tilde t\ge t_0$ and define
\[   \rho_B(t,x) =  (t+t_1)^{-n\lambda}   \Big(\frac{A^2}4 - \frac{|x|^2}{2(t+t_1)^{2\lambda}}\Big)^{\frac1{m-1}}_+   \]
which satisfies (with $c$ from \eqref{firstupperbound})
\[ \rho_B(\tilde t,x) \ge  \varepsilon \tilde t^{-n\lambda}  \qquad \text{ for } |x| \le c {\tilde t}^{\lambda}\]
if
\[ \Big(\frac{\tilde t}{\tilde t+t_1}\Big)^{n(m-1)\lambda} \Big( \frac{A^2}4- c^2 (\frac{\tilde t}{\tilde t+t_1})^{2\lambda} \Big)
 \ge \varepsilon^{m-1}
\]
which holds if
\[     \tilde t \ll   t_1  \ll       \tilde t \varepsilon^{-\frac1{n\lambda}}.  \]
Moreover
\[ \rho_B(t,x) \ge \varepsilon  \qquad \text{ for }  \tilde t \le t \le  T,  |x|= r(t) \]
if
\[  \varepsilon^{m-1}  \le \frac1{(T+t_1)^{(m-1)n\lambda}} \Big(\frac{A^2}4  - r^2(T) \frac{ T^{2\lambda}}{(t_1+ T)^{2\lambda}} \Big)    \]
which holds if
\[ T = \Big[\frac{(\frac{A^2}{4})^{\frac1{2\lambda}} }{1-(\frac{A^2}4)^{\frac1{2\lambda}} } - \varepsilon^{m-1}
+ \varepsilon^{-\frac12 - (\frac{m-2}{m-1})_+} t^{-\lambda} + \varepsilon^{-\frac1n}   t^{-\frac{2\lambda}n} \Big] t_1  \]
with obvious modifications if $n=1$ or $n=2$.

But then $\rho_B(T,.)$ is supported in $\overline{B_R(0)}$  with
\[ R =    D (T+\delta t \varepsilon^{-\frac2n})^\lambda \le A (T^{\lambda}+ \epsilon^{m-1} + \varepsilon^{-\frac12 -(\frac{m-2}{m-1})_+} t^{-\lambda}+ \varepsilon^{-\frac1n}  t^{-\frac{2}{n} \lambda}  ) \]
Now we optimize $\varepsilon$ and arrive at
\[
\supp \rho(T,.) \subset \supp \rho_B(T,.) \subset B_{R(T)+ cT^{-\beta}}.
 \]
for some positive constant $\beta$ depending on $n$ and $m$ resp.
\begin{equation}
\supp \tilde \rho(s) \subset B_{1+ e^{-\beta s/\lambda}}(0).
\end{equation}

\subsubsection{Lower bounds}

With this information we derive a lower bound on $c_0$ by
\[ \int (A^2-y^2)_+^{\frac1{m-1}} dy =   \int \widetilde \rho(s,y) dy \le
\int (c_0-|y|^2)_+^{\frac{1}{m-1}} dy + c\epsilon^{-\frac12-(\frac{m-2}{m-1})_+} e^{-s} + c\epsilon  e^{-\beta s/\lambda}   \]
and hence
\[   (A^2- c e^{-\beta s/\lambda}  - \frac12 |y|^2 )_+ \le \tilde \rho^{m-1} \le (A^2+  c e^{-\beta s/\lambda}  -\frac12 |y|^2)_+ \]
for some positive $c$ and $\beta$. We obtain

\begin{proposition} Let $\rho_0$ belong to the class $\mathcal C$. There exists \ $t_0(n,m)>0$ \ such that for $t \ge t_0M^{m-1}R^{\frac1\lambda}$ the solution $\rho(t,x)$ is sandwiched between
\begin{equation}\label{form.t}
 t^{-n\lambda} \Big(A^2-c_1t^{-\lambda\beta}  -  \frac{\lambda|x|^2}{2t^{2\lambda}}  \Big)^{\frac1{m-1}}_+       \le \rho(t,x)  \le  t^{-n\lambda} \Big(A^2+c_1t^{-\lambda\beta } - \frac{\lambda|x|^2}{2t^{2\lambda}}   \Big)^{\frac1{m-1}}
\end{equation}
where $A$ is given in \eqref{relmass}.  Moreover, the free boundary of $\rho(t,x)$ is contained in an annulus with radii $R_{\pm}(t)$ such that
\[ R_B t^{n\lambda} \le R_-(t)\le  R_+(t)\le R_B t^{n\lambda}(1+c_3t^{-\beta}).\]
The constants $t_0, A, c_1$ and $c_3$ are uniform for the class $\mathcal C$.
\end{proposition}

The proof that we have done assumes $M=1$ and $R=1$, and uses rescaled
variables, the usual variables are restored via formula
\eqref{cont.scal}. For general $u_0\in \mathcal C$ with $M,R>0$ use
the scaling transformation \eqref{scaling.form} with $L=1/R$,
$A=1/(MR^n)$, $x_0=0$, $t_0=0$, hence $1/C=M^{m-1}R^{2+n(m-1)}$, which
is the time scale factor.

\subsection{Flatness for large time}
We proceed with the proof of Theorem \ref{thm2}. We start from a time
large enough so that \eqref{form.t} applies near the time $t=t_1$
after rescaling.  An elementary geometric estimate shows that, as a
consequence of \eqref{form.t}, there exists $c$ depending only on $n$
and $m$ so that if $t_1\ge T\ge 2t_0$, $(t_1,x_1) \in \partial
\mathcal{P}(\rho)$ and $\mathbf{a} = x_1/|x_1|$ then
\[
\begin{split}
\rho_{tw}\Big(t,x; \mathbf{a},1, - cT^{-\frac{\beta}2}    \Big)   \le &
T^{\frac{\beta}{2(m-1)}} t_1^{-\frac{1}{m-1}}  \rho (t_1+  t_1^{2\lambda-1} T^{-\frac{\beta}2}    t   ,  x_1 +t_1^\lambda T^{-\frac{\beta}2} x) \\ &   \le \rho_{tw}(t, x; \mathbf{a},1, cT^{-\frac{\beta}2})
\end{split}
 \]
 for $ |x| \le 1 $ and $-1\le T \le 0$. Hence, if $ \delta:= c
 T^{-\frac{\beta}2}\le \delta_0$ with the constants of Theorem
 \ref{th-main} we apply that Theorem to
\[ \overline \rho(t,x) = T^{\frac{\beta}{2(m-1)}} t_1^{-\frac{1}{m-1}}
\rho (t_1+ t_1^{2\lambda-1} T^{-\frac{\beta}2} t , x_1 +t_1^\lambda
T^{-\frac{\beta}2} x) \]
and conclude that at near the new time $t=0$
the solution is $C^\infty$ regular in $x$ and $t$ up to the free
boundary which is $C^\infty$ hypersurface very close to the rescaled
Barenblatt free boundary. This gives the proof of the regularity part
of Theorem \ref{thm2}.

\subsection{Changing dependent and independent variables for large time}

 To complete the proof we have to obtain the rate of convergence.
We rewrite the equation with self-similar coordinates by parametrizing the graph of the pressure $\tilde \rho^{m-1}(s,y)$ as
\[ \Phi: B_{\sqrt{2}}(0) \times (0,\infty) \to (xu, u^2(1-|x|^2/2) ).
\]
Then $u=const$ is the Barenblatt solution and $y$ and $v= \tilde \rho^{m-1}$
are related to $x$ and  $u$  by
\[    y = xu \qquad v = u^2(1-|x|^2/2) \]
  and a tedious calculation gives
 \begin{equation}\label{global}   u_s -(m-1) \sum_i \partial_i\Big[ \Big(1-\frac{|y|^2}2\big)F^i \Big] -(m-2)
y_i F^i =0
\quad \text{ on } \
[T,\infty) \times B_{\sqrt{2}}(0),
 \end{equation}
where
\[ F^i = \partial_{x_i} u - \frac{x_i |\nabla u|^2}{u+x_k\partial_k u }. \]
It is not hard to see that \eqref{precisesupport} is a consequence of \eqref{precise} which in these coordinates becomes
\begin{equation}\label{precisetrans}    |u-\sqrt{A_{m,n}}| \le c e^{-2s}. \end{equation}  It remains to prove \eqref{precisetrans}.

The linearization at $u=1$ is
\begin{equation}\label{linearized} \dot u_s - (m-1) \sum_{i}\partial_i  \Big[\Big(1-\frac12|x|^2\Big) \partial_i \dot u\Big]  - (m-2) x_i \partial_i \dot u= 0\qquad \text{ on }
[0,\infty) \times B_{\sqrt{2}}(0).  \end{equation}

The mass is given by
\[
\begin{split}
 M= &\int \tilde \rho dx \\
=& \int u^{\frac{2}{m-1}}(x)\Big(1-\frac12|x|^2\Big)^{\frac1{m-1}} \det ( u \delta_{ij}+ \partial_i u x_j )_{1\le i,j\le n}
\\ = & \int u^{n+\frac2{m-1}-1} (u + x_k \partial_k u) \Big(1-\frac12|x|^2\Big)^{\frac1{m-1}} dx
\\ = & \int \Big[u^{n+\frac2{m-1}}+ \frac1{n+\frac2{m-1}} x_k \partial_k u^{n+
\frac2{m-1}} \Big]     \Big(1-\frac12|x|^2\Big)^{\frac1{m-1}}
\\ = & 2\lambda \int_{B_{\sqrt2}(0)}  u^{n+\frac2{m-1}}    \Big(1-\frac12 |x|^2\Big)^{\frac{2-m}{m-1}} dx
\end{split}
\]
and the entropy by
\[
\begin{split}
H = &  \int \tilde \rho \big( \frac{\tilde \rho^{m-1}}m + \frac12 |y|^2 \big) dy
\\ = & \frac1m \int_{B_{\sqrt{2}}(0)} u^{n+1+\frac2{m-1}}(u+x_k \partial_k u ) (1+(m-1)|x|^2/2)        \Big( 1- \frac12 |x|^2\Big)^{\frac1{m-1}} dx
\\ = & \frac2{(n+2)(m-1)+2}  \int_{B_{\sqrt2}(0)}  u^{n+2+\frac2{m-1}}   \Big(1-\frac12|x|^2\Big)^{\frac{2-m}{m-1}} dx.
 \end{split}
\]

\subsection{Relating entropy dissipation and an energy for $u$}

\begin{lemma} Suppose that $\lambda s \ge \ln T(n,m)$. Then $u$ is well defined
and the following estimates hold
\begin{equation}\label{pointwise}    |\sqrt{A_{n,m}} - u| \le C e^{-\beta \lambda s},
\end{equation}
\begin{equation}\label{derivative}
 |\nabla u | \le C e^{-\beta\lambda s}.
 \end{equation}
Moreover,
\begin{equation}\label{decay}   \int \Big(1-\frac{|y|^2}2\Big)^{\frac1{m-1}}|  \nabla u(s,.)  |^2 \le c  \int \tilde \rho \Big|y+\nabla \tilde \rho^{m-1}\Big|^2 dy\le c e^{-2s}\,.
\end{equation}
\end{lemma}

\begin{proof}
The pointwise bound is equivalent to  \eqref{form.t}. Elementary
geometric considerations show that
\[
|\nabla \tilde \rho^{m-1} - y| \sim |\nabla u |
\]
 if $|\nabla v -y |
\le \frac12 $ or $|\nabla u| \le \frac12$. In an $\sqrt{\delta_0}$
neighborhood of the free boundary this estimate is a consequence of
the regularity estimates of Theorem \ref{th-main}. In the interior it
follows from the bound on entropy dissipation rate
\eqref{def.entrdiss} and Moser's $L^\infty$ bounds. We obtain
\eqref{decay} from the decay of the entropy dissipation.
\end{proof}

Since $u$ can be written as a solution to
\begin{equation}\label{linearized2} u_t - (m-1)
  \Big(1-\frac{|x|^2}2\Big)^{\frac{2-m}{m-1}} \sum_{i}\partial_i
  \Big[\Big(1-\frac{|y|^2}2\Big)^{\frac{1}{m-1}} a^{ij} \partial_j u
  \Big] = 0\qquad \text{ on }
[0,\infty) \times B_{\sqrt{2}}(0)  \end{equation}
we obtain for $s\ge s_0+1$ the existence of a constant $\kappa$ so that
\begin{equation} \label{dist} \Vert u(s)- \kappa \Vert^2_{sup} \le c
  \Vert u(s-1) - \kappa \Vert^2_{L^2} \le c \int (1-|x|^2/2) |\nabla
  u(s-1)|^2 dx \le c e^{-2s}
\end{equation}
where we used Moser's estimate for the first and Poincar\'e's inequality for the second inequality. Clearly  $\kappa = \sqrt{A_{nm}}$.  Arguing similarly
in the interior and at the boundary we arrive at
\begin{equation}  \Vert \partial^\alpha_x \partial_s^k u(s,x) \Vert_{sup} \le c  e^{-s} . \end{equation}

\subsection{The spectrum of the linearization \eqref{linearized}}

The spectrum of the elliptic operator in \eqref{linearized} determines
the convergence rate. We summarize the relevant results of Seis,
Proposition 6 in \cite{Seis2}.  The operator can be written as
\[ u \to Lu = -(m-1)\Big( 1- \frac{|x|^2}2\Big)^{\frac{2-m}{m-1}} \nabla\Big[ \Big( 1-
\frac{|x|^2}2\Big)^{\frac1{m-1}}\nabla u\Big] = 0. \]
Let $H$ be the Hilbert
space $L^2(B_{\sqrt{2}}(0), (1-|x|^2/2)^{\frac{2-m}{m-1}} )$.  Then
$L$ is the positive semi-definite operator on $H$ defined by the
quadratic form
\[
E(u)= (m-1) \int \Big(1-\frac{|x|^2}2\Big)^{\frac{1}{m-1}} |\nabla u|^2 dx.
\]
The Hilbert space $H^1$ defined by
\[
\Vert u \Vert^2_{H^1}= E(u) + \Vert u \Vert_{H}^2
\]
embeds compactly into $H$ and hence the spectrum of $L$ consists of a
sequence of eigenvalues tending to $\infty$. Both norms are invariant
under rotations and we can diagonalize it into spherical harmonics of
degree $l$. For each degree we obtain a sequence of simple eigenvalues
tending to infinity:
\[ \lambda_{lk} = (l+2k) + (m-1) k(2k+2l+n-2) \]
where $l$ and $k$ are
nonnegative integers. The first eigenvalues and eigenfunctions are:

\begin{enumerate}
\item $\lambda_{00} = 0$, the eigenspace  is spanned by $u=1$.

\smallskip

\item $\lambda_{10} = 1$, the eigenspace is  spanned by $u=x_i$.

\smallskip

\item $\lambda_{20} = 2$, $ u =  x^t A x$ for a traceless matrix $A$.

\smallskip

\item $\lambda_{l0} = l$, $ u$ is a harmonic polynomial of degree $l$.

\smallskip

\item $\lambda_{01} = \lambda^{-1}$, $u= |x|^2- 2n(m-1)\lambda$ with a crossover at $m= 1+\frac1n$ and a second at $m=1$. This mode corresponds to a time shift
in the original variables.

\smallskip

\item $\lambda_{02} = 2\lambda^{-1} + 4 (m-1)$ with radial eigenfunction
\[
u(x)=  |x|^4+ \frac{4(m-1)(n+2)}{2-\lambda_{02}}\Big(|x|^2 + \frac{4(m-1)n}{\lambda_{02} }\Big).  \]
\end{enumerate}

This is the first relevant eigenvalue in the radial case (resp. the case with vanishing harmonic moments).

\subsection{Consequences for stability}

We now discuss their relevance for the nonlinear dynamics.

\smallskip

\noindent (1) $ \dot u = 0 $ with eigenvalue $0$ is the least stable
mode. It corresponds to changes of mass. We eliminate this mode by
fixing the mass of the solution.

\smallskip

\noindent (2) A mismatch in the center of mass corresponds to the modes
  \[
   \dot u = y_i
   \]
with eigenvalue $-1$. This implies the convergence rate
$t^{-\lambda}$ in self-similar coordinates.  Note that the integral
$ \int \rho(x) x_i dx $ is conserved by the PME flow. In
particular, if the center of mass is $0$ initially then this
remains so for all time. This can be done by a mere displacement of
the axes.  Equation \eqref{global} can now be written as a
perturbation of the linear equation    \eqref{linearized2}.

\smallskip

\noindent (3) The next eigenvalue is $-2$ and hence
\[
u =  A_{n,m} + \sum_{i=1}^n a_i y_i e^{-s} + O(e^{-(2-\varepsilon)s} )
\]
in $L^2$. Since
\[
0 = \int y_i \tilde \rho(s,y) dy = \int a_i y_i^2 e^{-t} + O(e^{-(2-\varepsilon)s})\,,
\]
we conclude that $a_i=0$ and, as above
\[
\sup_{|\alpha| \le M} | \partial^\alpha u | \le  c e^{-(2-\varepsilon)s}.
\]
Thus,
\[
u_s - (m-1) (1-|y|^2/2)^{\frac{2-m}{m-1}}    \sum_{i}\partial_i  [(1-|y|^2/2)^{\frac{1}{m-1}} \partial_i u]  = O(e^{-(4-\varepsilon)s})   \qquad \text{ on }
[0,\infty) \times B_{\sqrt{2}}(0)\,,
\]
 and we expand into the next modes (in $L^2$)
 \[ u= ( x^t Ax e^{-2s} + h_3(x) e^{-3s} + (|x|^2 - 2n(m-1)
 \lambda)e^{-s/\lambda} + O( e^{-(4-\varepsilon) s })\,,
\]
where $A$ is  a traceless matrix, $h_3$ is a homogeneous harmonic polynomial of
 degree $3$. It is not hard to show that  this expansion
 holds in a smooth sense, and to determine the next term.

\smallskip

\noindent (4) One can \ determine $A$ and the coefficients $h_3$ from
the initial data as follows: If $f$ is a smooth function then
\[ \frac{d}{dt} \int \rho f (x) dx =  \int \rho^m \Delta f dx.
\]
In particular, if $h$ is a harmonic polynomial then $\int \rho h dx =0 $.
Then, $A$ is given by the harmonic second order moments.

\noindent (5) In the radial case there exists $t_0$ so that
\[
\left|  e^{n\lambda s} \rho(e^{s\lambda}+ t_0, \sqrt{\lambda}  e^{\lambda s} y)^{m-1}
- (A_{n,m} - |x|^2/2)_+ \right| \le c  e^{-(2 + 4\lambda (m-1))s}.
\]

\medskip

\noindent {\bf Acknowledgment.} {H. Koch and C. Kienzler have been supported
by Hausdorff Center for Mathematics in Bonn.
The last author was supported by projects MTM2011-24696 and MTM2014-52240-P (Spain).}


\medskip

\begin{thebibliography}{99}

\bibitem {Ang} { S.~B. Angenent.} {\sl Large-time asymptotics of the porous media equation},  in ``Nonlinear Diffusion Equations and their Equilibrium States I'',
(Berkeley, CA, 1986), W.-M. Ni, L. A.  and J. Serrin eds., MSRI
Publ. {\bf 12}, Springer Verlag, Berlin, 1988.


\bibitem{AA95} {S.~B. Angenent, D.~G. Aronson.}
{\sl The focusing problem for the radially symmetric porous medium
     equation},
Comm. Partial Differential Equations {\bf 20 } (1995), 1217--1240.

\bibitem{Aro86} D. G. Aronson. {\sl \lq \lq The Porous Medium Equation''}. In: Nonlinear Diffusion Problems (A. Fasano and M. Primicerio, eds.), pp 1 - 46. Springer-Verlag, Berlin, 1986.

\bibitem{AC83} { D.~G. Aronson, L.~A. Caffarelli.}
{\sl The initial trace of a solution of the porous medium equation,} Trans.
     Amer. Math. Soc. 280 (1983), 351--366. 

\bibitem{AG93} { D.~G. Aronson,   J.~A. Graveleau.}
{\sl Self-similar solution to the focusing problem for the porous
medium equation,} {European J. Appl. Math.} {\bf 4} (1993), no.
1, 65--81.

\bibitem{ArVaz87}  D.~G. Aronson, J.~L.~V\'azquez.
{\sl Eventual $C^{\infty}$-regularity and concavity of flows in
one-dimensional porous  media},
 Archive Rat. Mech. Anal. {\bf 99} (1987), 329--348.

\bibitem{BCP84} { P. B\'enilan,  M.~G. Crandall, M. Pierre.}
{\sl Solutions of the porous medium in $\R^n$ under optimal conditions
on the initial  values.} {Indiana Univ. Math. Jour.}  {\bf 33}
(1984), pp. 51--87.


\bibitem{Caff87} L.~A. Caffarelli. {\sl A Harnack inequality approach to the regularity of free boundaries. Part I: Lipschitz free boundaries are $C^{1,\alpha}$}, Rev. Mat. Iberoamericana 3 (1987), no.2, 139--162.

\bibitem{Caff89} L.~A. Caffarelli. {\sl A Harnack inequality approach to the regularity of free boundaries. Part
II: Flat free boundaries are Lipschitz}, Comm. Pure Appl. Math. 42 (1989), no. 1, 55--78.


\bibitem{CaffFr} L.~A. Caffarelli. A.Friedman. {\sl  Regularity of the
free boundary of a gas flow in an $n$-dimensional porous medium,}
{Indiana Univ. Math. J.} {\bf 29} (1980), 361--391.

\bibitem{CVW87} L.~A. Caffarelli, J. L. V\'azquez, N. I. Wolanski. {\sl Lipschitz Continuity of
Solutions and Interfaces of the $N$-Dimensional Porous Medium Equation},
Indiana Univ. Math. J. {\bf 36} (1987), no. 2, 373--401.

\bibitem{CW90} L.~A. Caffarelli, N.~I. Wolanski. {\sl $C^{1,\alpha}$ Regularity
of the Free Boundary for the N-Dimensional Porous Media Equation,}
Comm. Pure Appl. Math. {\bf 43} (1990), no. 7, 885--902.

\bibitem{CaTo00}
{J.~A.~Carrillo, G.~Toscani.} {\sl Asymptotic $L^1$-decay of solutions of the porous medium
equation to self-similarity},
 Indiana Univ. Math. J. {\bf 49} (2000), 113--141.

\bibitem{Da1}  E.~B. Davies. {\sl ``Heat kernels and spectral theory''}, Cambridge Univ. Press, 1989.

\bibitem{FriedmanPDE} A. Friedman. {\sl \lq\lq Partial Differential Equations, }''
Holt, Rinehart, and Winston, New York (1969).

\bibitem{Kienz13} C. Kienzler. {\sl Flat Fronts and Stability for the Porous Medium Equation}, Dissertation, 2013.

\bibitem{Kienz14} C. Kienzler. {\sl Flat Fronts and Stability for the Porous Medium Equation}, arxiv.org 1403.5811, 2014.

\bibitem{Koc99} H. Koch. {\sl Non-Euclidean Singular Integrals and the Porous Medium Equation. Habilitation}, Ruprecht-Karls-Universit\"at Heidelberg, 1999.

\bibitem{LV03} K.-A. Lee,  J.~L. V\'azquez. {\sl  Geometrical
properties of solutions of the Porous Medium Equation for large
times,}  Indiana Univ. Math. J. {\bf 52} (2003), no. 4, 991--1016.

\bibitem{LUS75} O. A. Ladyzhenskaya, N. N. Ural'ceva, and V. A. Solonnikov. {\sl \lq \lq Linear and Quasi-Linear Equations of Parabolic Type''}. American Mathematical
Society, Providence, 1975.


\bibitem{Seis1}  C. Seis. {\sl Long-time asymptotics for the porous medium equation: The spectrum of the linearized operator},  J. Differential Equations {\bf 256 } (2014), no. 3, 1191--1223.




\bibitem{Seis2}  C. Seis. {\sl Invariant manifolds for the porous medium
equation}, arXiv:1505.06657.

\bibitem{Vsym} {J. L. V\'azquez},
{\sl Sym\'etrisation pour $u_t=\Delta\varphi(u)$ et applications,}
C. R. Acad. Sc. Paris {\bf  295} (1982), pp. 71--74.

\bibitem{vazquez1}  J.~L. V\'azquez. {\sl Asymptotic behaviour for the
Porous Medium Equation posed in the whole space}.
 J. Evol. Equ.  {\bf 3} (2003),  no. 1, 67--118.


\bibitem{Vsmooth} J.~L. V\'{a}zquez. {\sl \lq\lq Smoothing and Decay Estimates
 for Nonlinear Diffusion Equations. Equations of Porous Medium Type''.}
Oxford University Press, Oxford Lecture Series, 2006.


\bibitem{vazquezPME} J.~L. V\'{a}zquez. {\sl \lq\lq The porous medium equation.  Mathematical theory''}.    Oxford Mathematical Monographs. The Clarendon Press, Oxford University
Press, Oxford, 2007. 



\end{thebibliography}
\end{document}